\documentclass[11pt,reqno, a4paper]{amsart}

\usepackage{dsfont} 
\usepackage{mathtools}
\usepackage{amsmath}
\usepackage{amssymb}
\usepackage{mathrsfs}
\usepackage{tikz} 
\usetikzlibrary{arrows} 
\usepackage{wrapfig}

\numberwithin{equation}{section}

\newcommand{\B}{{\mathbb{B}}}
\newcommand{\C}{{\mathbb{C}}}

\newcommand{\HH}{{\mathbb{H}}}
\newcommand{\I}{{\mathbb{I}}}

\newcommand{\M}{{\mathbb{M}}}

\renewcommand{\P}{{\mathbb{P}}}

\newcommand{\R}{{\mathbb{R}}}

\newcommand{\T}{{\mathbb{T}}}
\newcommand{\Z}{{\mathbb{Z}}}

\newcommand{\Aa}{{\mathcal{A}}}

\newcommand{\Dd}{{\mathcal{D}}}

\newcommand{\Gg}{{\mathcal{G}}}
\newcommand{\Hh}{{\mathcal{H}}}
\newcommand{\Ii}{{\mathcal{I}}}

\newcommand{\Ll}{{\mathcal{L}}}
\newcommand{\Mm}{{\mathcal{M}}}
\newcommand{\Nn}{{\mathcal{N}}}

\newcommand{\Ss}{{\mathcal{S}}}
\newcommand{\Tt}{{\mathcal{T}}}
\newcommand{\Uu}{{\mathcal{U}}}
\newcommand{\Vv}{{\mathcal{V}}}

\newcommand{\Xx}{{\mathcal{X}}}

\newcommand{\swann}{\mathcal{U}(N)}

\newcommand{\euler}{\mathcal{X}_{0}}  
\newcommand{\imag}{\mathrm{\mathbf{i}}}

\newcommand{\sst}{\scriptscriptstyle}

\newcommand{\mbf}[1]{\mathbf{#1}}

\newcommand{\scr}[1]{\mathscr{#1}}
\newcommand{\mf}[1]{\mathfrak{#1}}

\newcommand{\norm}[1]{\left\| #1 \right\|}
\newcommand{\abs}[1]{\left\lvert #1 \right\rvert}
\newcommand{\pair}[1]{\left\langle #1 \right\rangle}

\newcommand{\eqst}[1]{\begin{equation*} #1 
                      \end{equation*}}
\newcommand{\eq}[1]{\begin{equation} #1
                    \end{equation}}
\newcommand{\alst}[1]{\begin{align*} #1 
                      \end{align*}}
\newcommand{\al}[1]{\begin{align} #1
                      \end{align}}

\theoremstyle{plain}
\newtheorem{thm}{Theorem}[section]
\newtheorem{lem}[thm]{Lemma}
\newtheorem*{lemma*}{Lemma}
\newtheorem{prop}[thm]{Proposition}
\newtheorem{cor}[thm]{Corollary}

\theoremstyle{definition}

\theoremstyle{definition}
\newtheorem{defn}{Definition}

\newtheorem*{note}{Note}
\newtheorem{rmk}{Remark}

\theoremstyle{remark}
\newtheorem*{question*}{Question}

\DeclareMathOperator{\id}{id}
\DeclareMathOperator{\End}{End}

\DeclareMathOperator{\grad}{grad}

\DeclareMathOperator{\img}{\mf{Im}}
\DeclareMathOperator{\re}{Re}
\DeclareMathOperator{\tr}{tr}
\DeclareMathOperator{\alt}{\rm \bf Alt}
\DeclareMathOperator{\sym}{\rm \bf Sym}

\usepackage{hyperref}
\hypersetup{
   colorlinks = true,
    linkcolor = blue,
    citecolor = blue,   
}

\allowdisplaybreaks

\begin{document}

\title[Hypersymplectic manifolds and associated geometries]{Hypersymplectic manifolds and associated geometries}

\author[V. Thakre]{Varun Thakre}

\address{\hspace*{-4mm}International Centre for Theoretical Sciences (ICTS-TIFR), Hesaraghatta Hobli, Bengaluru 560089, India}

\email{varun.thakre@icts.res.in}


\date{Revised on \today }


\begin{abstract}
We investigate an obstruction for hypersymplectic manifolds equipped with a free, isometric action of ${\rm SU}(1,1)$. When the obstruction vanishes, we show that the manifold is a metric cone over a split 3-Sasakian manifold. Furthermore, if the action of ${\rm SU}(1,1)$ is also proper, then the hypersymplectic manifold fibres over a para-quaternionic K\"ahler manifold. We conclude the article with some examples for which the obstruction vanishes. In particular, we show that the moduli space to Nahm-Schmid equations admits a fibration over a para-quaternionic  K\"ahler manifold.
\end{abstract}

\maketitle


\section{Introduction}

A hypersymplectic manifold is a $4n$-dimensional pseudo-Riemannian manifold, equipped with a metric of neutral signature $(2n, 2n)$, and whose holonomy is contained inside the symplectic group ${\rm Sp}(2n, \R)$. It can be viewed as a pseudo-Riemannian analogue of hyperK\"ahler manifolds. Hypersymplectic geometry appears naturally in the study of integrable systems \cite{bartocci-mencattini04}, string theory \cite{hull98} - where it is also known by \emph{Kleinian geometry} - and gauge theory \cite{brr18, roeser14}. The terminology ``\emph{hypersymplectic}'' is due to Hitchin \cite{hit90}.

A powerful tool for constructing hypersymplectic manifolds is the hypersymplectic quotient construction, which is an adaptation of the Marden-Weinstein construction in symplectic geometry. However, in contrast with the hyperK\"ahler situation, more ofthen than not, the hypersymplectic structure on the quotient, is degenerate \cite{djs05}.

Another way of obtaining hypersymplectic manifolds is via an adaptation of Swann's bundle construction in hyperK\"ahler geometry \cite{swann91}. Starting with a quaternionic K\"ahler manifold of positive scalar curvature, say $N$, Swann's construction produces a bundle, $\swann\rightarrow N$ with a typical fibre $\HH^{\ast}/\Z_2$, whose total space carries a hyperK\"ahler structure. It is possible to carry over this construction to the pseudo-Riemannian case \cite{djs05}. In order to do so, one needs a para-quaternionic K\"ahler manifold; i.e., a $4n$-dimensional pseudo-Riemannian manifold, whose holonomy is contained inside the group 
\eqst{
{\rm Sp}(2n, \R)\cdot {\rm Sp}(2, \R) \, = \, {\rm Sp}(2n, \R)\times_{\pm 1} {\rm Sp}(2, \R).
}
They can be thought of as pseudo-Riemannian analogues of quaternionic K\"ahler manifolds. Starting with a para-quaternionic K\"ahler manifold $N$, the construction produces a bundle $\swann \rightarrow N$, with a typical fibre $\B^{\ast}/\Z_2$, where $\B^{\ast}$ is the space of non-zero split quaternions with non-zero norm. The total space $\swann$ carries a hypersymplectic structure. Both the para-quaternionic K\"ahler and hypersymplectic geometries are Einstein. Additionally, the latter is also Ricci-flat. Para-quaternionic K\"ahler manifolds are characterised by the existence of a closed 4-form, whereas hypersymplectic manifolds are equipped with family of symplectic 2-forms.

In this article, we study a more general picture. Namely, \emph{given a hypersymplectic manifold, when does it admit a para-quaternionic K\"ahler quotient?} Our basic observation is that, the total space of a Swann bundle over a para-quaternionic K\"ahler manifold admits a free, proper, isometric action of ${\rm Sp}(2,\R)\cong {\rm SU}(1,1)$, which is an analogue of the permuting ${\rm Sp}(1)$-action on hyperK\"ahler manifolds. For such hypersymplectic manifolds, we construct two maps $\rho_2: M \rightarrow S^4(B)$ - where $B$ is the standard representation of ${\rm SU}(1,1)$ on the vector space of split quaternions $\B$ - and $\rho_0: M \rightarrow \R_{>0}$. If $\rho_2$ vanishes, we show that $\rho_0$ is a \emph{hypersymplectic potential}. The level-sets of $\rho_0$ carry a split 3-Sasakian structure and the metrics on different level-sets are homothetic. In particular, the hypersymplectic manifold can be thought of as a metric cone over a split 3-Sasakian manifold. Additionally, if the action of ${\rm SU}(1,1)$ is also proper, we show that the quotient of a level-set of $\rho_0$, by $\rm{SU}(1,1)$, is a para-quaternionic K\"ahler manifold.

This approach is analogous to that of Boyer, Galicki and Mann \cite{bgm93} for hyperK\"ahler manifolds with permuting ${\rm Sp}(1)$-action.

Split 3-Sasakian structures were introduced by Swann, J{\o}rgensen and Dancer in \cite{djs05} and have  also been studied by Caldarella and Pastore \cite{cp09}, where they are referred to as ``\emph{mixed 3-Sasakian structures}". The authors show that any split 3-Sasakian structure is necessarily Einstein.

We give two examples of hypersymplectic manifolds which have the ${\rm SU}(1,1)$ symmetry, with vanishing obstruction. First example is that of hypersymplectic manifolds, obtained via hypersymplectic reduction of flat-space. We show that the Swann-bundle construction commutes with the quotient construction, which produces a family of examples of the theory. The results complement the work of Swann, J{\o}rgensen and Dancer in \cite{djs05}. 

The second example is that of the moduli space of Nahm-Schmid equations defined on the interval $[0,1]$  \cite{brr18}. The solutions to Nahm-Schmid equations exist for all times. As a result, it is possible to define a scaling action on the moduli space of solutions. The moduli space also carries a free, proper, permuting action of ${\rm SU}(1,1)$. Therefore, topologically, it has the structure of a metric cone over a split 3-Sasakian manifold. The quotient of the latter by the ${\rm SU}(1,1)$ action is a paraquaternionic K\"ahler manifold. In other words, the moduli space can be expressed as the total space of a Swann bundle over a paraquaternionic K\"ahler manifold.

\section{Acknowledgements}
\noindent The author wishes to express his heartfelt thanks to Prof. Andrew Dancer and Dr. Markus R\"oser for many helpful discussions and to Prof. Stefan Ivanov for the references \cite{iv04, iz05}. The author also wishes to thank the anonymous referee for many helpful comments, especially on the content in Sub-section 6.2 on Nahm-Schmid equations. This work originated, in part, through discussions with Prof. Nuno Rum\~{a}o during the program  \emph{Integrable systems in Mathematics, Condensed Matter and Statistical Physics} (Code: ICTS/integrability2018/07) at the International Centre for Theoretical Sciences (ICTS-TIFR). The author wishes to thank the organizers of the program for their support.


\section{Brief introduction to split quaternionic geometry}

The space of \emph{split quaternions} is a 4-dimensional vector space $\B$, spanned by $(1, \imag, \mbf{s}, \mbf{t})$, satisfying the following relations 
\eq{
\label{eq: split quaternionic relations}
\imag^2 = -1, \,\,\, \mbf{s}^2 = 1 = \mbf{t}^2, \,\,\, \imag\cdot \mbf{s} = \mbf{t} = -\mbf{s}\cdot \imag, \,\,\, \imag = \mbf{t}\cdot \mbf{s}.
} 
Like quaternions, the vector space of split quaternions comes equipped with a multiplication operation, which gives it a structure of an associative algebra.

\noindent The vector space carries a natural inner product defined by $\pair{p, q}_{\B} = \re p \,\overline{q}$, where $p = p_0 + \imag\,p_1 + \mbf{s}\,p_2 + \mbf{t}\,p_3$ and $\overline{p} = p_0 - \imag\,p_1 - \mbf{s}\,p_2 - \mbf{t}\,p_3$. Since the metric is not positive definite, it only makes sense to talk about the (isotropic) quadratic form (or the `norm-square') $\norm{p}^2:=\pair{p, p}_{\B}$, associated with the neutral signature metric. Given two split-quaternions $p$ and $q$, we have
$\norm{p\,q}^2 = \norm{p}^2\cdot \norm{q}^2,$
showing that the quadratic form is multiplicative. Note that if $p$ is a point in $\mf{Im}(\B)$, then $\norm{p}^2 = p\cdot \overline{p} = -p^2 = p_1^2 - p_2^2 - p_3^2 \in \R$.

Unlike the quaternionic algebra, the split quaternion algebra contains non-trivial zero divisors. Moreover, the elements $\imag, \mbf{s}, \mbf{t} \in \mf{Im}(\B)$ are not the only elements with length $\pm 1$. Elements with norm $1$ are parametrized by the 2-sheeted hyperboloid $x_1^2 - x_2^2 - x_3^2 = 1$, while those with norm $-1$ are parametrized by the 1-sheeted hyperboloid $x_1^2 - x_2^2 - x_3^2 = -1$. Any triple $\{ \imag, \mbf{s}, \mbf{t} \}$ satisfying \eqref{eq: split quaternionic relations} defines a \emph{split-quaternionic structure} on $\B$.

Let $U = \{q\in \B ~| \norm{q}^2 \neq 0 \}$ be the set of all units in $\B$. This is clearly a multiplicative group. The subset of $U$ consisting of all elements $q$ with $\norm{q}^2 = 1$ forms a non-compact topological group ${\rm SU}(1,1)$. This is the special unitary group of all complex $2 \times 2$ matrices $g$ that satisfy 
\begin{enumerate}
\item Unimodularity, i.e, $\det g = 1$ 
\item Pseudo-unitary condition: i.e, $g^{\ast} \, J \, g = J$ where $J = \begin{bmatrix}
1 & 0 \\
0 & -1
\end{bmatrix}$
\end{enumerate}
In particular, any element $g \in {\rm SU}(1,1)$ has the form $g = \begin{bmatrix}
\alpha & \beta \\
\beta^{\ast} & \alpha^{\ast}
\end{bmatrix}$, where $\alpha$ and $\beta$ are complex numbers subject to the condition $\abs{\alpha}^2 - \abs{\beta}^2 = 1$. The Lie algebra of ${\rm SU}(1,1)$ is 3-dimensional
\eqst{
\mf{su}(1,1) \, = \, \text{Span} \left\lbrace \, 
\begin{bmatrix}
\imag & 0 \\
0 & -\imag
\end{bmatrix}, 
\begin{bmatrix}
0 & 1 \\
1 & 0
\end{bmatrix},
\begin{bmatrix}
0 & \imag \\
-\imag & 0
\end{bmatrix}
\right\rbrace \cong \mf{Im}(\B).
}
\paragraph*{\textbf{Relation with \texorpdfstring $~{\rm SO}(1,2)$:}}
Consider the 3-dimensional Lorentz group ${\rm SO}(1,2)$. This is the group of transformations of the 3-dimensional Minkowski space $\M^3$, with determinant $1$, that preserves the quadratic form. The group acts transitively on the 1-sheeted and 2-sheeted hyperboloids and also on the cone $\norm{x}^2=0$.

Alternatively, if we consider the \emph{pseudo-sphere}
\eq{
\label{eq: pseudo-sphere}
\scr{H} = \{ \, q \in \mf{Im}(\B) \,\, | \, \norm{q}^2 = \pm 1 \, \} \cong \{ \, (x_1, x_2, x_3) \in \R^3 \,\, | \, x_1^2 - x_2^2 - x_3^2 = \pm 1 \, \}
}
then, the hyperboloids can be thought of as unit spacelike and timelike vectors in the $\scr{H}$ and the ${\rm SO}(1,2)$-action on $\scr{H}$ is then an analogue of the standard action of ${\rm SO}(3)$ on the 2-sphere $S^2$.

The group ${\rm SO}(1,2)$ is disconnected and has two connected components. We denote by ${\rm SO}^+(1,2)$ the identity component. 
Identifying $\M^3$ with the imaginary split-quaternions $\mf{Im}(\B)$,
it is easily seen that the adjoint action of ${\rm SU}(1,1)$ on $\M^3$ preserves the quadratic form, the pseudo sphere $\scr{H}$ and the null-cone. Therefore the linear transformations corresponding to the adjoint action of the elements of ${\rm SU}(1,1)$ belong to the identity component ${\rm SO}^+(1,2)$. This gives a homomorphism from ${\rm SU}(1,1)$ to ${\rm SO}^+(1,2)$ with kernel $\pm 1$; i.e., ${\rm SU}(1,1)/\pm 1 \cong {\rm SO}^+(1,2)$, similar to the homomorphism between ${\rm SU}(2)$ and ${\rm SO}(3)$.

\subsection{Modules over split quaternions}
\label{subsec: modules over split quaternions}

\noindent Consider the \emph{left $\B$ module} $\B^n \cong \R^{4n}$, equipped with the split quaternionic structure $I, S, T$, given by
\eq{
\label{eq: paraquat str. on paraquat modules}
I\,(q) \, = \, q \,\overline{\imag}, \,\,\,\,\,\, S\,(q) \, = \, q\, \mbf{s}, \,\,\,\,\,\, T\,(q) \, = \, q\, \mbf{t}. 
}
\begin{rmk}
Any element $q \in \scr{H}$ determines a product or a complex structure on $\B^n$. To see this, let $\lambda: \, \mf{Im}(\B) \longrightarrow \End\,(\B^n)$ denote the algebra homomorphism
\eqst{
\lambda(q) = \I_q \, := \, q_1 \, I \, + \, q_2 \, S\, + \, q_3 \, T.
}
Then $q$ determines a product or a complex structure, depending on whether $\norm{q}^2 = \mp 1$. In other words, the pseudo-sphere $\scr{H}$ parametrizes the complex and the product structures on $\B^n$.
\end{rmk}
The module $\B^n$ inherits the natural inner product 
\eqst{
\pair{\alpha, \beta} \, := \, \re\, \left(\beta \, \overline{\alpha}^T \right), \,\,\,\, \alpha, \, \beta \in \B^{n}.
}
The automorphism group of $\B^n$, given by
\eqst{
{\rm Sp}(n, \B) \, := \, \{ \, A \in M_n(\B) \,\, | \,\, A \overline{A}^T \, = \, \id \, \} \cong {\rm Sp}(2n, \R),
}
is nothing but the automorphism group of the symplectic vector space $(\R^{2n}, \omega_{\sst \R^{2n}})$. The Lie algebra of ${\rm Sp}(n, \B)$ is given by
\eqst{
\mf{sp}(n, \B) \, := \, \{ \, A \in M_n(\B) \,\, | \,\, A \, + \, \overline{A}^T \, = \, 0 \, \}.
}
Note that for $n=1$, we have the isomorphism ${\rm Sp}(1, \B) \cong {\rm SU}(1,1)\cong {\rm SL}(2, \R)$ and so, we can identify the Lie algebra $\mf{sp}(1, \B) = \mf{Im}(\B)$. 

Consider the action of the group ${\rm Sp}(n, \B) \times {\rm Sp}(1, \B)$ on $\B^n$, given by 
\eqst{
(A, \, \xi) \cdot q \longmapsto A\, q \,\overline{\xi}.
}
Let $\Lambda_{\scr{H}}$ denote the image of $\mf{Im}(\B)$ under the map $\lambda$. It is easy to see that the action of ${\rm Sp}(n, \B)$ is isometric and the induced action on $\Lambda_{\scr{H}}$ \emph{preserves} $\Lambda_{\scr{H}}$, pointwise. On the other hand, the ${\rm Sp}(1,\B)$-action is isometric, but the induced action on $\Lambda_{\scr{H}}$ is, pointwise, nothing but the standard action of ${\rm SO}^+(1,2)$ on $\scr{H}$. Indeed, for any $\xi \in {\rm Sp}(1, \B)$,
\eqst{
\xi \cdot (\I_q (h)) \, = \, \xi \cdot (h \,\overline{q}) \, = \, h\,  \overline{q}\,\overline{\xi} \, = \, h\, \overline{\xi}\, (\xi \, \overline{q}\, \overline{\xi}) \, = \, \I_{\text{Ad}_{\xi^{-1}} \,q}\cdot (\xi \,\cdot h).
}


\section{Hypersymplectic manifolds}

Let $(M, g_{\sst M}, I, S, T)$ be a $4n$-dimensional pseudo-Riemannian manifold, endowed with a triple of endomorphisms $I, S, T$, satisfying the split quaternionic relations \eqref{eq: split quaternionic relations} and a metric of neutral signature $(2n, 2n)$, that is compatible with the split quaternionic structure
\eqst{
g_{\sst M}(IX, \, IY) \, = \, g_{\sst M}(X, \, Y), \,\,\,\,\, g_{\sst M}(SX, \, SY) \, = \, -g_{\sst M}(X, Y) \, = \, g_{\sst M}(TX, \, TY).
}
The split quaternionic structure allows us to define the following 2-forms on $M$
\eqst{
\omega_1 (X, \, Y) \, := \, g_{\sst M} (IX, \, Y), \,\,\,\,\, \omega_2(X, \, Y) \, := \, g_{\sst M}(SX, \, Y), \,\,\,\,\, \omega_3(X, \, Y) \, := \, g_{\sst M}(TX, \, Y).
}

If each of the above 2-forms are closed, the manifold $M$ is called a \emph{hypersymplectic manifold}. Using Hitchin's arguments for the hyperK\"ahler manifolds, one can show that the structures $I, S, T$ are integrable; i.e., they are parallel with respect to the Levi-Civita connection. As a result, the holonomy group of $M$ reduces to ${\rm Sp}(n, \B)$.

The endomorphisms $S$ and $T$ are called \emph{product structures}. This is because the integrability of these structures implies that the manifold $M$ locally looks like a product $M^+ \times M^-$, where $\pm$ denotes the eigenvalues $\pm 1$ of $S$ ot $T$ and $TM^{\pm}$ denotes the corresponding eigenspaces. In fact, every element of the 1-sheeted hyperboloid $x_1^2 - x_2^2 - x_3^2 = -1$ determines a product structure as
\eqst{
(x_1, x_2, x_3) \longmapsto x_1\, I \, + \, x_2 \, S \, + \, x_3\, T.
}
Such structures are also known by \emph{paracomplex structures} in literature. On the other hand, $M$ also has a family of \emph{pseudo-K\"ahler structures}, which are parametrized by the 2-sheeted hyperboloid $y_1^2 - y_2^2 - y_3^2 = 1$ as
\eqst{
(y_1, y_2, y_3) \longmapsto y_1\, I \, + \, y_2 \, S \, + \, y_3\, T.
}

A \emph{pseudo-K\"ahler structure} on a manifold $M$ is a complex structure $I$ on $M$ along with a pseudo-Riemannian metric $g_{\sst M}$, such that the metric is compatible with $I$ and the 2-form $\omega_{\sst M}(\cdot, \cdot) = g_{\sst M}(I(\cdot), \cdot)$ is closed. In other words, a pseudo-K\"ahler structure is just a pseudo-Riemannian analogue of K\"ahler structure.

In some cases, it is possible to explicitly construct a family of examples of hypersymplectic manifolds. Ivanov and Zamkovoy \cite{iz05} constructed a hypersymplectic structure on Kodaira-Thurston (properly elliptic) surfaces. Andrada and Salamon, in \cite{as05}, show that if there exists a complex product structure on a real Lie algebra $\mf{g}$; i.e., a pair ${I, S}$ of complex structure and a product structure, then, it induces a hypersymplectic structure on the complexification $\mf{g}^{\C}$. In \cite{iv04}, Ivanov and Tsanov showed that the manifolds underlying the Lie groups ${\rm SL}(2m - 1, R)$ and ${\rm SU}(m, m - 1)$ carry a complex product structure which induces a hypersymplectic structure on their complexifications.

For a hypersymplectic manifold $(M, g_{\sst M}, I, S, T)$, let $\Ii\subset \End(TM)$ denote the trivial 3-dimensional sub-bundle spanned by $(I, S, T)$. Any covariantly constant endomorphism $\I \in \Ii$ can be thought of as a map with values in $\mf{sp}(1, \B)^{\ast}$, using the algebra homomorphism
\eqst{
\lambda: \, \mf{sp}(1, \B) \longmapsto \End(TM), \,\,\,\,\, h \longmapsto
\I_{h} \, := \lambda(h), \,\,\,\,\,\,\, h \in \mf{sp}(1, \B) = \img(\B).
}
Similarly, the associated symplectic 2-forms can be clubbed into a single $\mf{sp}(1, \B)^{\ast}$-valued 2-form as
\eqst{
\pair{\omega, h} \, := g_{\sst M}(\,\I_{h}\cdot, \, \cdot\,) \, = \, h_1 \, \omega_1 \, \imag + \, h_2 \, \omega_2 \, \mbf{s} + \, h_3 \,
 \omega_3 \, \mbf{t}.
}
Let $\Lambda_{\scr{H}} \subset \Ii$ denote the image of $\scr{H} \subset \mf{sp}(1,\B)$ under the map $\lambda$. In particular, $\Lambda_{\scr{H}}$ consists of all the product and complex structures on $M$.



\subsection{Permuting actions}

Consider the fundamental 4-form 
\eqst{
\Omega \, = \, \omega_1\wedge\omega_1 - \omega_2 \wedge \omega_2 - \omega_3 \wedge \omega_3.
}
The form is globally defined on $M$. The stabilizer group ${\rm St}_{\Omega} \in \text{Isom}(M, g_{\sst M})$ of $\Omega$ is a sub-group of the group of isometries that preserves each symplectic 2-form $\omega_i$. The induced action of ${\rm St}_{\Omega}$, on $\Lambda_{\scr{H}}$, determines the homomorphism
\eq{
\label{eq: permuting corr.}
{\rm St}_{\Omega} \rightarrow {\rm Sp}(1,\B)/\pm 1 \cong {\rm SO}^+(1,2).
}
The kernel of this homomorphism is the group of \emph{hypersymplectic isometries}, whose induced action on $\Lambda_{\scr{H}}$, pointwise, fixes $\Lambda_{\scr{H}}$. 

\begin{defn}
An isometric action of the group ${\rm Sp}(1,\B)$ on a hypersymplectic manifold $M$ is said to be \emph{permuting}, if the induced action on $\Lambda_{\scr{H}}$, is the standard action of ${\rm SO}^+(1,2)$ on $\scr{H}$.  In other words, the action is induced via the epimorphism ${\rm Sp}(1,\B) \rightarrow {\rm St}_{\Omega} \rightarrow {\rm SO}^+(1,2)$.
\end{defn}

Henceforth, without loss of generality, we will assume that $M$ admits a free, permuting, effective action of the group ${\rm Sp}(1,\B)$. The arguments that follow are an adaptation of the representation theoretic arguments in \cite{bgm93, victor}. 

\noindent Let $K^M_{\xi}$ denote the fundamental vector field on $M$ corresponding to $\xi \in \mf{sp}(1, \B)$. Define the following operators:
\eqst{
\iota: \otimes^{p} \mf{sp}(1, \B)^{*} \otimes \Omega^{q}(M) \longrightarrow \mf{sp}(1, \B)^{*} \otimes^p \mf{sp}(1, \B)^{*} \otimes \Omega^{q-1}(M), \,\,\,\,\,\,\,\,\,\,\, \ \pair{\iota(\alpha), \, \xi} = \iota_{\sst K^{M}_{\xi}}\, \alpha
}
and 
\alst{
\Ll_{\mf{sp}(1,\B)}:\otimes^{p} \mf{sp}(1, \B)^{*} \otimes \Omega^{q}(M) \longrightarrow \mf{sp}(1, \B)^{*}\otimes^p \mf{sp}(1, \B)^{*} \otimes \Omega^{q}(M), \,\,\,\,\,\,\,\,\,\,\,\,\, \pair{\Ll_{\mf{sp}(1,\B)}\,\alpha, \, \xi} = \mathcal{L}_{\sst K^{M}_{\xi}} \,\alpha.
}
Then Cartan's formula $\Ll_{\mf{sp}(1, \B)} = d\,\iota_{\mf{sp}(1, \B)} + \iota_{\mf{sp}(1, \B)} \,d$ is easily verified.

\begin{lem}[\cite{bgm93, victor}]
\label{lem: identity of omega}
For the $\mf{sp}(1, \B)^{\ast}$-valued 2-form $\omega$ we have
\eq{
\label{eq: identity of omega}
\Ll_{\mf{sp}(1, \B)} \, \omega = 2\,\omega.
}
\end{lem}

\begin{proof}

We first verify that $\omega$ is ${\rm Sp}(1, \B)$-equivariant. Let $q \in {\rm Sp}(1, \B)$ and $\xi \in {\rm Sp}(1, \B)$. Then for the vector fields $V$, $W$ on $M$
\alst{
\langle q^{*} \omega, \,\xi \rangle(V,W) 
= g_{\sst M}(\,\I_{\xi}\,(q_{*}V), \,q_{*}W) = g_{\sst M}(\,q_{*}^{-1}\,\I_{\xi}\,(q_{*}V), \, W) = \langle \omega, \, \text{Ad}_{q^{-1}}\,(\xi) \rangle (V,W).
}

Consider $\xi, \xi' \in \mf{sp}(1, \B)$ such that $\xi\xi' \neq 0 \neq \xi'\xi$. Then, using the identity above, we get:
\alst{
\pair{\Ll_{\mf{sp}(1, \B)} \,\omega, \, \xi\otimes\xi'}
& = \Ll_{K^{M}_{\xi}} \, \omega_{\,\xi'} = \frac{d}{dt} (L_{\exp(-t \xi)})^{*} \omega_{\,\xi'} \bigg\rvert_{t=0} = \frac{d}{dt} \omega_{\,\text{Ad}_{\exp(t \xi)}\, \xi'} \bigg\rvert_{t=0} = \pair{\omega, \,[\xi, \xi']}.
}

Note that this implies that $\Ll_{\mf{sp}(1, \B)} \, \omega$ is a $\Lambda^2(\mf{sp}(1, \B))^{\ast}$-valued 2-form on $M$. We have the isomorphism 
\eqst{
[\cdot,\cdot]: \Lambda^{2} \mf{sp}(1, \B) \longrightarrow \mf{sp}(1, \B)
}
given by
\eqst{
\imag \wedge \mbf{s} \longmapsto [\imag, \mbf{s}] = 2 \mbf{t}, \,\,\,\,\,
\mbf{s} \wedge \mbf{t} \longmapsto [\mbf{s}, \mbf{t}] = 2 \overline{\imag}, \,\,\,\,\,
\mbf{t} \wedge \imag \longmapsto [\mbf{t}, \imag] = 2 \mbf{s}.
}

This induces an isomorphism of between $\Lambda^{2} \mf{sp}(1, \B)^{\ast}$ and $\mf{sp}(1, \B)^{\ast}$. Therefore, we can think of $\Ll_{\mf{sp}(1,\B)} \, \omega$ as a $\mf{sp}(1, \B)^{\ast}$-valued 2-form on $M$. It is in this sense that we write the equality in \eqref{eq: identity of omega}. A straight forward computation using the isomorphism $[\cdot, \cdot]$ now shows that $\Ll_{\mf{sp}(1, \B)}\, \omega = 2 \,\omega$.

\end{proof}

Define the 1-form $\gamma = \frac{1}{2}\,\iota_{\mf{sp}(1, \B)}\, \omega \in \mf{sp}(1, \B)^{*} \otimes \mf{sp}(1, \B)^{*} \otimes \Omega^{1}(M)$. More precisely 
\eqst{
\pair{\gamma, \, \xi \otimes \xi'} \, = \, \frac{1}{2}\, g_{\sst M}(\I_{\xi} K^M_{\xi'}, \cdot).
}

\noindent The tensor product $\mf{sp}(1, \B)^{*} \otimes \mf{sp}(1, \B)^{*}$ splits into a direct sum of sub-representations $S^{2} \left(\mf{sp}(1, \B)^{*} \right) \oplus \Lambda^{2}\left(\mf{sp}(1, \B)^{*} \right)$. The symmetric part further decomposes into a direct sum of the trace and the traceless component. Consequently,
\eq{
\label{eq: main clebsch-gordon decomposition}
\mf{sp}(1, \B)^{*} \otimes \mf{sp}(1, \B)^{*} = \R \oplus \Lambda^{2}\left(\mf{sp}(1, \B)^{*} \right) \oplus S_{0}\left(\mf{sp}(1, \B)^{*} \right).
}
This is the Clebsch-Gordon decomposition. Correspondingly, the 1-form $\gamma$ decomposes into three components 
\eq{
\label{eq: clebsch-gordon decomposition}
\iota_{\mf{sp}(1, \B)}\, \omega \, = \, (\gamma_0, \gamma_1, \gamma_2).
}
From Lemma \ref{lem: identity of omega}, it follows that $d\gamma = 2\omega$, since $d\omega = 0$. However, note that the right hand side belongs to the $\mf{sp}(1, \B)^{\ast} \cong \Lambda^{2}\left(\mf{sp}(1, \B)^{*} \right)$. This implies that $d\gamma_0 = 0 = d\gamma_2$ and $d\gamma_1 = 2\omega$. 

\begin{prop}
Let $M$ be a hypersymplectic manifold with a permuting action of the group ${\rm Sp}(1,\B)$. Let $\Gg \subset \Lambda^2(M)$ denote the trivial sub-bundle spanned by $\omega_1, \omega_2, \omega_3$. Then, the de-Rham cohomology class of any symplectic 2-form in $\Gg$ vanishes. In particular, $M$ can never be compact.
\end{prop}

\begin{lem}[\cite{victor}]
\label{lem: identity of gamma_1}
The map $\gamma_{1}$ satisfies the following identity 
\eq{ 
\label{eq: lie derivative of gamma1}
\Ll_{\mf{sp}(1, \B)} \, \gamma_{1} \, = 2\, \gamma_{1}.
}
\end{lem}
\begin{proof}
The proof is identical to that of Lemma \ref{lem: identity of omega}.
\end{proof}

Following the approach in \cite{victor, henrik} for the hyperK\"ahler case, we now show that $\gamma_0$ and $\gamma_2$ are exact. Define $\rho := \iota_{\mf{sp}(1, \B)}\, \gamma_1 \in \Omega^0 \left(M, \, \mf{sp}(1, \B)^{\ast} \otimes \mf{sp}(1, \B)^{\ast} \right)$. Corresponding to the decomposition \eqref{eq: clebsch-gordon decomposition}, the map $\rho$ has 3 components:
\eqst{
\rho \, = \, (\rho_0, \, \rho_1, \, \rho_2).
}
Denote by $\text{\rm \bf Alt}$, the projection of $\mf{sp}(1, \B)^{*} \otimes \mf{sp}(1, \B)^{*}$ to the alternating part $\Lambda^{2}\, \mf{sp}(1, \B)^{*}$ and by $\text{\rm \bf Sym}_0$, the projection of $\mf{sp}(1, \B)^{*} \otimes\mf{sp}(1, \B)^{*}$ to the traceless, symmetric part $S^{2}\,(\mf{sp}(1, \B)^{*})$. Then, the identity \eqref{eq: lie derivative of gamma1} can be written as
\eqst{
\Ll_{\mf{sp}(1, \B)} \, \gamma_{1} \, = \, \alt \left(\iota_{\mf{sp}(1, \B)} \, \omega \right).
}
Therefore we can write
\eq{
\label{eq: expression for d gamma_1}
d\,\iota_{\mf{sp}(1, \B)} \, \gamma_1 \, = \, \Ll_{\mf{sp}(1, \B)} \, \gamma_1 - \iota_{\mf{sp}(1, \B)}\, d\gamma_1 \, = \, \alt \left(\iota_{\mf{sp}(1, \B)} \, \omega \right) - \iota_{\mf{sp}(1, \B)}\, \omega.
}
It follows that
\eqst{
d\rho_0 \, = \, d \left(\frac{1}{3} \tr \left(\iota_{\mf{sp}(1, \B)}\, \gamma_1 \right) \right) \, = -\, \frac{1}{3} \tr \iota_{\mf{sp}(1, \B)} \, \omega \, = \, \gamma_0
}
and 
\eqst{
d\rho_2 \, = \, d \left(\sym_0\left(\iota_{\mf{sp}(1, \B)} \, \gamma_1\right)\right) \, = \, -\sym_0 \left(\iota_{\mf{sp}(1, \B)} \, \omega\right) = \gamma_2.
}
In particular, $\gamma_0$ and $\gamma_2$ are exact.

\subsubsection{\bf Potentials}

\subsubsection*{Para-K\"ahler manifold} A para-K\"ahler manifold $(M, g_{\sst M})$, is a pseudo-Riemannian manifold, endowed with a metric compatible, parallel, skew-symmetric endomorphism $T \in \End(TM)$ satisfying $T^2 = 1$.


\noindent Suppose that $(M, g_{\sst M}, T)$ is a para-K\"ahler manifold. Let $\omega$ be the para-K\"ahler 2-form given by
\eqst{
\omega \, = \, g_{\sst M}(T(\cdot), \, \cdot).
}
For any 1-form $\alpha$ on $M$, define $T\, \alpha (v):= -\alpha(Tv)$ for $v \in TM$. A \emph{para-K\"ahler potential} is a smooth function $\rho: M \rightarrow \R$ such that $\frac{1}{2}\,\, dId\rho = \omega$. Note that any hypersymplectic manifold is also a para-K\"ahler manifold in many different ways. The para-K\"ahler structures are parametrized by the 1-sheeted hyperboloid. 

Define $\kappa(\xi) := -\iota_{K^M_{\xi}} \, \gamma_1 (\xi)$ for any $\xi \in \mf{sp}(1, \B)$ such that $\norm{\xi} \neq 0$.

\textbf{Case 1:} ~Suppose that $\xi \in \mf{sp}(1, \B)$ is such that $\xi^2 = 1$. Then the stabilizer of $\I_{\xi}$ is a sub-group ${\rm SO}^+(1,1) \subset{\rm Sp}(1,\B)$, consisting of $2 \times 2$ real matrices of the form $g = \begin{bmatrix}
a & b \\
b & a
\end{bmatrix}$ such that $a^2 - b^2 = 1$. Its Lie algebra is the vector space of real numbers $(\R, +)$. The group ${\rm SO}^+(1,1)$ preserves the symplectic $2$-form $\omega_{\xi}$. Moreover, the associated moment map is given by $\kappa(\xi)$. Indeed, this can be seen as follows:
\eqst{
d\,\kappa(\xi) \, = -\, (d\,\iota_{\mf{sp}(1, \B)} \, \gamma_1) (\xi, \xi) \, = \,  (\iota_{\mf{sp}(1, \B)}\, \omega)\, (\xi, \xi) \, = \, \iota_{K^M_{\xi}}\, \omega_{\xi}.
}
Here we have used the identity \eqref{eq: expression for d gamma_1}.
\smallskip

\textbf{Case 2:} ~Suppose that $\xi \in \mf{sp}(1, \B)$ is such that $\xi^2 = -1$. Then $\I_{\xi}$ defines a complex structure. Let ${\rm U}(1) \in {\rm Sp}(1, \B)$ denote the stabilizer of $\I_{\xi}$. Then, by the same argument as above, $\kappa(\xi)$ defines the moment map with respect to the ${\rm U}(1)$-action, for the pseudo-K\"ahler 2-form $\omega_{\xi}$.

\begin{prop}
\label{prop: para kahler and kahler potentials}
Let $M$ be a hypersymplectic manifold with a permuting ${\rm Sp}(1,\B)$-action. Let $\xi, \xi' \in \mf{sp}(1, \B)$ with $\xi \xi' \neq 0 \neq \xi'\xi$. Then the following holds
\begin{enumerate}
\setlength{\itemsep}{2mm}
\item If their squares are $-1$ and they are perpendicular, then, $-\kappa(\xi')$ is the pseudo-K\"ahler potential for the pseudo-K\"ahler 2-form $\omega_{\xi}$.
\item If their squares are $-1$ and $1$ respectively, then, $\kappa(\xi')$ is the pseudo-K\"ahler potential for the pseudo-K\"ahler 2-form $\omega_{\xi}$.
\item If their squares are $1$ and $-1$ respectively OR both the squares are $1$ and if $\xi$ and $\xi'$ are perpendicular, then, $\kappa(\xi')$ is a ``para K\"ahler potential" for the symplectic 2-form $\omega_{\xi}$.
\end{enumerate}
\end{prop}

\begin{proof}
\begin{enumerate}
\setlength{\itemsep}{2mm}
\item The proof follows from the following straight-forward computation
\alst{
-\frac{1}{2} \left(d\,\I_{\xi}^{\ast}\, d\right) (-\kappa(\xi')) 
& = \frac{1}{2}d\,\I_{\xi}^{\ast}\,g_{\sst M}(\, \I_{\xi'}K^M_{\xi'}, \, \cdot) = \frac{1}{2}d\,g_{\sst M}\left(\, \I_{\xi'}K^M_{\xi'}, \, \I_{\xi}(\cdot) \right)\\
& = -\frac{1}{2}d\,g_{\sst M}\left(\, \I_{\xi}\I_{\xi'}K^M_{\xi'}, \, \cdot \right) = -\frac{1}{4} \, d \, \iota_{K^M_{\xi'}} \, \omega\,([\xi, \xi'])\\
& = -\frac{1}{4}\, \Ll_{K^M_{\xi'}}\, \omega\,([\xi, \xi']) = \frac{1}{4} \, \omega\, ([\xi', \, [\xi, \, \xi']]) = \omega\, (\xi).
}
The last equality follows from the fact that $[\xi', \,[\xi, \, \xi']] = 4\,\xi$.
\item Proof follows from the following computation, which is a slight variation of the one above
\alst{
-\frac{1}{2} \left(d\,\I_{\xi}^{\ast}\, d\right) (\kappa(\xi')) 
& = -\frac{1}{2}d\,\I_{\xi}^{\ast}\,g_{\sst M}(\, \I_{\xi'}K^M_{\xi'}, \, \cdot) = \frac{1}{2}d\,g_{\sst M}\left(\, \I_{\xi'}K^M_{\xi'}, \, \I_{\xi}(\cdot) \right)\\
& = \frac{1}{2}d\,g_{\sst M}\left(\, \I_{\xi}\I_{\xi'}K^M_{\xi'}, \, \cdot \right) = \frac{1}{4} \, d \, \iota_{K^M_{\xi'}} \, \omega\,([\xi, \xi'])\\
& = \frac{1}{4}\, \Ll_{K^M_{\xi'}}\, \omega\,([\xi, \xi']) = -\frac{1}{4} \, \omega\, ([\xi', \, [\xi, \, \xi']]) = \omega\, (\xi).
}
The last equality follows from the fact that $[\xi', \,[\xi, \, \xi']] = -4\,\xi$. Thus $-\kappa(\mbf{s})$ and $-\kappa(\mbf{t})$ are K\"ahler potentials for $\omega_1$.
\item By arguments identical to the ones above, we have 
\eqst{
-\frac{1}{2} \left(d\,\I_{\xi}^{\ast}\, d\right) (\kappa(\xi')) = \omega\, (\xi).
}
Therefore $\kappa(\mbf{s}), \,\kappa(\mbf{t})$ are para K\"ahler potentials for $\omega_3$ and $\omega_2$, respectively.
\end{enumerate}
\end{proof}

Define
\eqst{
\Xx \in \mf{sp}(1, \B)^{\ast} \otimes \mf{sp}(1, \B)^{\ast} \otimes \Gamma(M, TM), \,\,\,\,\, \Xx (\xi, \xi') = \I_{\xi} K^M_{\xi'}.
}
Note that if $\norm{\xi}^2 = 0$, then $g_{\sst M}(\Xx(\xi, \xi'), \, \Xx(\xi, \xi') ) = 0$. In other words, $\I_{\xi}$ maps $K^M_{\xi'}$ to a null-vector, for any $\xi' \in \mf{sp}(1,\B)$.
Owing to the Clebsch-Gordon decomposition \eqref{eq: main clebsch-gordon decomposition}, the map $\Xx$ splits into three parts: $\Xx_{0}$, $\Xx_{1}$, $\Xx_{2}$, given by
\alst{
& \Xx_{0} \, = \, -\frac{1}{3} \tr \, \Xx \, = \, -\frac{1}{3} \left(I \, K^M_{\xi_1} \, + \, S\,K^M_{\xi_2} \, + \, T\, K^M_{\xi_3} \right) \in \Gamma(M, TM)\\
& \Xx_{1}([\xi, \xi']) \, = \, \frac{1}{2} \left( \I_{\xi}\, K^M_{\xi'} \, - \,\I_{\xi'}\, K^M_{\xi} \right) \in \mf{sp}(1, \B)^{*} \otimes \Gamma(M, TM)\\
& \Xx_{2}(\xi, \xi') = -\mathcal{X}_{0} \langle \cdot, \cdot \rangle_{\B} - \frac{1}{2} \left(  \I_{\xi}\, K^M_{\xi'} \, + \,\I_{\xi'}\, K^M_{\xi} \right)\in S^{2}_{0}\left(\mf{sp}(1, \B)^{*} \right)\otimes \Gamma(M, TM).
}
Clearly then, $\gamma_i = \frac{1}{2}\,g_{\sst M}(\Xx_i, \cdot)$ and therefore $\Xx_i$ are the gradient vector fields for $\rho_i$.

\begin{defn}[Hypersymplectic potential]

Given a hypersymplectic manifold $(M, g_{\sst M}, I, S, T)$, a smooth function $\rho_0: M \rightarrow \R$ is said to be a \emph{hypersymplectic potential} if it is simultaneously a potential for $I, S, T$.

\end{defn}

\begin{lem}
\label{lem: euler vf}
Let $M$ be a hypersymplectic manifold with a free, permuting ${\rm Sp}(1,\B)$-action and assume $\Xx_2 = 0$ (equivalently, $\rho_2 = 0$). Then, $\rho_0$ is the hypersymplectic potential and we have
\eqst{
\euler \, = \, -\I_{\xi}K^M_{\xi} \,\,\,\, \text{for all} \,\,\,\, \xi \in \mf{sp}(1,\B), \, \text{such that} \,\,\,\, \norm{\xi}^2 \neq 0 \,\,\,\, \text{and} \,\,\,\, g_{\sst M} (\euler, \euler) = \rho_0 > 0.
}
Moreover, for any $\xi$ with $\norm{\xi}^2 = \pm 1$, the vector field $\euler$ is independent of $\xi$.
\end{lem}
\begin{proof}
Since $\Xx_2 = 0$, it implies that $\rho_2$ is constant on connected components. However, since $\rho_2$ is ${\rm Sp}(1, \B)$-equivariant, it must be identically zero on each of the connected components. In particular, $\rho_2\equiv 0$. Conversely, if $\rho_2 \equiv 0$, then it follows that $\Xx_2 = 0$, since $\Xx_2 = \grad (\rho_2)$. Consequently, since
\eqst{
\Xx_2(\xi, \xi) \, = \, -\euler - \frac{1}{2}\left(\I_{\xi}K^M_{\xi} + \I_{\xi}K^M_{\xi} \right) \,\,\,\,\,\,\, \text{it implies} \,\,\,\,\,\,\, \euler = - \I_{\xi}K^M_{\xi}.
}

Clearly, if $\norm{\xi} = \pm 1$, the $\euler$ is independent of $\xi$. Since $\euler$ is the gradient vector field of $\rho_0$, it follows that $\frac{1}{2} \,g_{\sst M} (\euler, \euler) = \rho_0$. From Proposition \ref{prop: para kahler and kahler potentials} we have $\rho_0 = \kappa(\mbf{i}) = -\kappa(\mbf{s}) = -\kappa(\mbf{t})$. Thus $\rho_0$ is the hypersymplectic potential. In particular, $d\,\I_{\xi}\, d \rho_0 = \epsilon\,\omega_{\xi}$, where $\epsilon = \pm 1$, according to whether $\norm{\xi}^2 = \mp 1$.
\end{proof}

\begin{note}
\label{note: note on norms of fundamental vector fields}
Let $\{\, \xi_1, \xi_2, \xi_3 \, \}$ be the basis of $\mf{sp}(1, \B)$. If $\Xx_2 = 0$, then, the above Lemma says that
\eqst{
\rho_0 \, = \, g_{\sst M} (K^M_{\xi_1}, \, K^M_{\xi_1}) \, = \, -g_{\sst M} (K^M_{\xi_2}, \, K^M_{\xi_2}) \, = \, -g_{\sst M} (K^M_{\xi_3}, \, K^M_{\xi_3}).
}
It is important to note here that since the vector fields $K^M_{\xi_i}$ generate the free action of ${\rm Sp}(1, \B)$ on $M$, the norm $g_{\sst M} (K^M_{\xi_1}, \, K^M_{\xi_1})$ must be positive, while, the norms $g_{\sst M} (K^M_{\xi_2}, \, K^M_{\xi_2})$ and $g_{\sst M} (K^M_{\xi_3}, \, K^M_{\xi_3})$ must be negative. 
Therefore we must have $\rho_0 = g_{\sst M}(\euler, \euler) > 0$.
\end{note}

The existence of a hypersymplectic potential on $M$ implies that the metric $g_{\sst M}$ is incomplete. The remainder of the section is dedicated to proving this and a few other consequences of the vanishing of the map $\rho_2$.

\begin{prop}
\label{prop: identities with euler vf}
Let Let $M$ be a hypersymplectic manifold with a free, permuting ${\rm Sp}(1, \B)$-action and assume that $\Xx_2 = 0$. Then the following holds
\begin{enumerate}
\item $\gamma_1 \, = \, \iota_{\euler}\, \omega$
\item $\Ll_{\euler}\, \gamma_1 \, = \, 2\, \gamma_1$
\item $\Ll_{\euler} \, \omega \, = \, 2\, \omega$
\item $\Ll_{\euler}\, \rho_0 \, = \, 2\, \rho_0$
\item $\Ll_{\mf{sp}(1, \B)} \, \euler \, = \, 0$
\end{enumerate}
\end{prop}
\begin{proof}
First, we make the following observation. Owing to Lemma \ref{lem: euler vf}, we have
\eqst{
I \euler \, = \, K^M_{\xi_1}, \,\,\,\,\, S \euler \, = \, - K^M_{\xi_2}, \,\,\,\,\, T \euler \, = \, - K^M_{\xi_3}.
}
\begin{enumerate}
\item Recall that $\gamma_1 = \alt\,\left(\iota_{\mf{sp}(1, \B)} \, \omega \right)$. Therefore,
\eqst{
\pair{\gamma_1, \,\imag} \, = \, \frac{1}{2}\pair{\alt\,\left(\iota_{\mf{sp}(1, \B)} \, \omega \right), \, \imag} \, = \, -\frac{1}{2}\pair{\iota_{\mf{sp}(1, \B)} \, \omega, \, \mbf{s}\otimes \mbf{t}} \, = -\, \iota_{K^M_{\xi_2}}\, \omega_3 = \iota_{\euler}\, \omega_1.
}
The last equality can be seen as follows
\eqst{
\iota_{K^M_{\xi_2}}\, \omega_3 \, = \, g_{\sst M}\left(T\,K^M_{\xi_2}, \, \cdot  \right) \, = \, g_{\sst M} \left(I\, S\,K^M_{\xi_2}, \, \cdot \right) \, = \, -g_{\sst M} \left(I\euler, \, \cdot \right) \, = \, -\iota_{\euler}\,\omega_1.
}
Similarly, one can show that
\eqst{
\pair{\gamma_1, \, s} \, = \, \iota_{\euler}\, \omega_2, \,\,\,\,\, \pair{\gamma_1, \, t} \, = \, \iota_{\euler}\, \omega_3.
}
Therefore, we have $\gamma_1 \, = \, \iota_{\euler}\, \omega$.

\item The second claim follows directly from the first one by observing that
\eqst{
\Ll_{\euler}\,( \iota_{\euler} \, \omega) \, = \, \iota_{\euler}\, \Ll_{\euler} \omega \, = \, 2 \, \iota_{\euler} \, \omega \, = \, 2\, \gamma_1.
}

\item  Consider the argument above Lemma \ref{lem: identity of omega}. We have shown that $d\gamma_1 = 2\, \omega$. But from the claim (1), it follows that $d\gamma_1 = \Ll_{\euler}\, \omega$. In conclusion, $\Ll_{\euler}\, \omega = 2\, \omega$.

\item Observe that
\eqst{
\Ll_{\euler}\, \rho_0 \, = \, \iota_{\euler}\, d\, \rho_0 \, = \, d\, \rho_0\,(\euler) \, = \, 2\, \rho_0.
}

\item This follows from the invariance of $\rho_0$ under the action of ${\rm Sp}(1,\B)$.
\end{enumerate}
\end{proof}

\begin{prop}
\label{prop: euler vf identity}
The gradient vector field of the hypersymplectic potential $\rho_0$ satisfies
\eqst{
\nabla\, \euler \, = \, \id_{TM}
}
where $\nabla$ is the Levi-Civita connection of the metric $g_{\sst M}$.
\end{prop}
To show this, we use the following result by Swann
\begin{prop}[\cite{swann91}]
\label{prop: swann's proposition}
Let $(N, I, g_{\sst N})$ be a K\"ahler manifold and $\nabla$ be the associated Chern connection. A smooth function $\rho_0: N \longrightarrow \R^+$ is a K\"ahler potential if and only if
\eq{
\label{eq: kahler potential identity}
\frac{1}{2}\left(\nabla^2_{X, Y}\, \rho_0 \, + \, \nabla^2_{IX, IY}\, \rho_0 \right)  \, = \, g_{\sst N} (X, Y).
}
\end{prop}
The above statement of the theorem holds even if the metric $g_{\sst N}$ is a  pseudo-K\"ahler metric.

\begin{proof}[\bf Proof of Lemma \ref{prop: euler vf identity}]
If we regard $M$ as a pseudo-K\"ahler manifold with respect to the complex structure $I$, then we have shown that $\rho_0$ is a K\"ahler potential for $I$. Then, from \eqref{eq: kahler potential identity} it follows that
\eqst{
\nabla(d\rho_0) (X, Y) \, + \, \nabla(d\rho_0) (IX, IY) \, = \, g_{\sst M}(X, Y)
}
Now consider
\alst{
g_{\sst M}(\nabla_X\, \euler, \, Y) \, 
& = \, \nabla_X\,(g_{\sst M}(\euler, \, Y)) \, - \, g_{\sst M}(\euler, \, \nabla_X Y) \\
& = \nabla_X\, (d\rho_0(Y)) \, - \, d\rho_0(\nabla_X Y) \\
& = \nabla_X (d\rho_0) (Y).
}
Therefore, 
\alst{
\nabla_X (d\rho_0) (Y) \, + \, \nabla_{IX} (d\rho_0) (IY) \, 
& = \, g_{\sst M}(\nabla_X\, \euler, \, Y) \, + \, g_{\sst M}(\nabla_{IX}\, \euler, \, IY)\\
& = g_{\sst M}(\nabla_X\, \euler, \, Y) \, - \, g_{\sst M}(I\nabla_{IX}\, \euler, \, Y)\\
& = g_{\sst M}(\nabla_X\, \euler - I\nabla_{IX}\, \euler, \, Y).
}
In Proposition \ref{prop: identities with euler vf}, we have shown that $\Ll_{\euler}\, \omega_1 = 2\, \omega_1$ and therefore $\Ll_{\euler} \, I = 0$. Using this, we have
\eqst{
\nabla_{IX} \, \euler = \nabla_{\euler} \, IX \, - \, \Ll_{\euler}\, IX \, = \, I\nabla_{\euler} \, X - I \Ll_{\euler}\, X \, = \, I\nabla_X \, \euler.
}
Plugging this in the previous equation, we get
\eqst{
\nabla_X (d\rho_0) \,(Y) \, + \, \nabla_{IX} (d\rho_0)\, (IY) \, = \, 2\, g_{\sst M}(\nabla_X\, \euler, \, Y).
}
From Proposition \ref{prop: swann's proposition}, it follows that $g_{\sst M}(\nabla_X\, \euler, \, Y) = g_{\sst M}(X, \, Y)$, for all $X, Y \in TM$. Therefore, $\nabla\, \euler = \id_{TM}$.

\end{proof}

\begin{cor}
\label{cor: charaterization of potential}
The hypersymplectic potential $\rho_0$ satisfies
\eqst{
\nabla^2 \rho_0 \, = \, g_{\sst M}.
}
\end{cor}
There are two consequences of the above corollary. First, the metric on $M$ cannot be complete (see \cite{yano70}). Second, the metrics from difference level-sets of $\rho_0$ are homothetic. 

In the section that follows, we will show that the level-sets of the hypersymplectic potential carry a \emph{split 3-Sasakian structure}.


\section{Split 3-Sasakian geometry}

We begin this section by introducing $\varepsilon$-Sasakian manifolds. These are the pseudo-Riemannian analogues of Sasakian manifolds with either a complex or a product structure on the leaves of the 1-dimensional foliation. If the leaves of the foliation are endowed with a complex structure ($\varepsilon = -1$), we call it a \emph{pseudo-Sasakian structure}. If they are endowed with a product structure ($\varepsilon = 1$), then we call the manifold as \emph{para Sasakian manifold}. This is slightly different than the conventional terminologies in the literature. However, it allows for a simultaneous treatment of both the cases. Let $\Ss$ be a smooth manifold, equipped with a $(1,1)$-tensor $\Phi$, a nowhere vanishing vector field $\xi$ and a 1-form $\eta$ metric dual to $\xi$, satisfying the following relations
\al{
\label{eq: cond. para contact metric str.}
\Phi\circ \Phi \, = \, \varepsilon\,\id \, - \, \varepsilon \, \eta \otimes \xi, \,\,\, \eta(\xi) \, = \, 1, \,\,\, \eta(\Phi(\cdot)) \, = \, 0 \, = \Phi(\xi).
}
If $\varepsilon= -1$, then $(\Ss, \Phi, \eta, \xi)$ is called an \emph{almost-contact manifold} and if $\varepsilon = 1$, it is called an \emph{almost-para contact manifold}. To give a simultaneous treatment, we will refer to $(\Ss, \Phi, \eta, \xi)$ as an $\varepsilon$-almost contact structure. Consider the Nijenuis tensor of $\Phi$, which is a $(2,1)$-tensor, defined as
\eqst{
N_{\Phi}(X, Y) \, = \, [\Phi(X), \, \Phi(Y)] \, + \, \Phi^2[X, \, Y] \, - \, \Phi([X, \, \Phi(Y)]) \, - \, \Phi([\Phi(X), \, Y]).
}
The $(\Ss, \Phi, \eta, \xi)$ is said to be \emph{normal}, if $N_{\Phi} = d\eta \otimes \xi$.
Suppose now that $\Ss$ is endowed with a pseudo-Riemannian metric $g_{\sst \Ss}$, such that
\eqst{
g_{\sst \Ss}(\Phi(X), \,\Phi(Y)) \, = \, \varepsilon \left(\, - g_{\sst \Ss}(X, Y)\, + \, \tau\,\eta(X)\,\eta(Y) \, \right), \,\,\,\, \text{and} \,\,\,\,  d\,\eta \,(X,Y) \, = \, \tau\,g_{\sst \Ss}(\Phi(X), Y),
}
where $\tau = g_{\sst \Ss}(\xi, \xi)$, then, the $\varepsilon$-almost contact structure is said to be an $\varepsilon$-\emph{para contact metric structure} and the metric $g_{\sst \Ss}$ is said to be \emph{compatible} with the $\varepsilon$-para contact structure. Additionally, if the structure is normal, the manifold $(\Ss, g_{\sst \Ss}, \Phi, \eta, \xi)$ is said to be a $\varepsilon$-\emph{Sasakian manifold}. In other words, an $\varepsilon$-Sasakian manifold is an $\varepsilon$-almost contact manifold, endowed with a compatible pseudo-Riemannian metric and the structure is normal.

\begin{prop}\cite{gml01}
\label{prop: para sasakian str.}
An $\varepsilon$-contact manifold $(\Ss, g_{\sst \Ss}, \Phi, \eta, \xi)$ is $\varepsilon$-Sasakian if and only if the following is satisfied
\eq{
\label{eq: para-sasakian str.}
\nabla_X \Phi \,(Y)\, = \, g_{\sst \Ss}(\Phi (X), \, \Phi(Y)) \, \xi \, + \, \tau\,\eta(Y)\cdot\Phi\circ \Phi(X).
}
where $\nabla$ is the Levi-Civita connection of $g_{\sst \Ss}$.
\end{prop}

\subsection{Split 3-Sasakian manifolds}

Suppose that $\Ss$ is a pseudo-Riemannian manifold of dimension $4n+3$, carrying a metric $g_{\sst \Ss}$ of signature $(2n+1, 2n-2)$. Additionally, suppose that $\Ss$ also carries a triple of orthogonal Killing vector fields $(X_1, X_2, X_3)$, of lengths $1, -1, -1$ respectively, satisfying
\eq{
\label{eq: 3-sasakian relations}
\frac{1}{2}\,[X_1, \, X_2] \, = \, X_3, \,\,\,\, \frac{1}{2}\, [X_2, \, X_3] \, = \, - X_1, \,\,\,\, 
\frac{1}{2}\,[X_3, \, X_1] \, = \, X_2.
}
Define the 1-forms $\eta_{i}(Y) = \tau_i\, g_{\sst \Ss}(X_i, Y)$, where $\tau_i = g_{\sst \Ss}(X_i, X_i)$. Then, $\eta_{i}(X_j) = \delta_{ij}$. Suppose that $\Ss$ is also endowed with endomorphisms $(\Phi_1, \Phi_2, \Phi_3)$, satisfying
\eq{
\label{eq:metric compatability for split sasakian}
g_{\sst \Ss}(\Phi_i(X), \,\Phi_i(Y)) \, = \, \varepsilon_i \, \left(\, - g_{\sst \Ss}(X, Y)\, + \, \tau_i \, \eta_i(X) \,\eta_i(Y)\, \right),
}
where, $(\varepsilon_1,\varepsilon_2, \varepsilon_3) = (1, -1, -1)$. Then, the structure $(\Ss, g_{\sst \Ss}, \{\Phi_i, \eta_i, \xi_i\}_{i=1}^3)$ is said to be an \emph{almost split 3-contact structure}. Additionally, if 
\eqst{
d\,\eta_i \, = \, \tau_i \, g_{\sst \Ss}(\Phi_i(\cdot), \cdot),
}
then the structure is said to be a \emph{metric split 3-contact structure}. 

\begin{defn}
A metric split 3-contact manifold $(\Ss, g_{\sst \Ss}, \{\Phi_i, \eta_i, X_i\}_{i=1}^3)$  is a split 3-Sasakian manifold if the structures are normal; i.e.,
\eqst{
N_{\Phi_i} \, = \, d\, \eta_i \otimes X_i, \,\,\,\, i = 1,2,3.
}
\noindent Equivalently, from Proposition \ref{prop: para sasakian str.}, we see that $(\Ss, g_{\sst \Ss}, \{\Phi_i, \eta_i, X_i\}_{i=1}^3)$  is a split 3-Sasakian manifold if
\eq{
\label{eq: para sasakian condition}
\nabla_X \, \Phi_i \, (Y) \, = \, \, g_{\sst \Ss}(\Phi_i (X), \, \Phi_i (Y)) \, X_i \, + \, \tau_i \,\eta_i (Y)\cdot\Phi_i \circ \Phi_i (X) , \,\,\,\,\, X, Y \in TS
}
\end{defn}
\noindent Note that $(\Ss, g_{\sst \Ss}, \Phi_1, \eta_1, \xi_1)$ is a pseudo-Sasakian manifold, whereas, for $i = 2,3$, $(\Ss, g_{\sst \Ss}, \Phi_i, \eta_i, X_i)$  is a para-Sasakian manifold.

\begin{thm}[\cite{cp09}, Thm. 4.1, 5.7 and Prop. 4.2]
\label{thm: split 3-sasakian}
Let $(\Ss, g_{\sst \Ss}, \{\Phi_i, \eta_i, \xi_i\}_{i=1}^3)$ be a split 3-contact manifold of dimension $4n + 3$. Then, $S$ is necessarily a split 3-Sasakian manifold. The metric $g_{\sst \Ss}$ is Einstein and has a constant scalar curvature equal to $(4n+2)(4n+3)$.
\end{thm}
\begin{rmk}
The split 3-Sasakian structure we discuss below is of the type $(1, -1, -1)$; i.e, the norms of the Reeb vector fields generating the split 3-Sasakian structures are $(1, -1, -1)$. The authors in \cite{cp09} refers to this as the ``\emph{negative mixed 3-Sasakian structure}". 
\end{rmk}

The simplest example of a split 3-Sasakian manifold is the positive pseudo-sphere in the split quaternionic module $\B^{n+1}$
\eqst{
S_+ \, = \, \{~ q \in \B^{n+1} ~ | ~ \norm{q}^2 \, = \, 1 ~\}.
}
The manifold $S_+$ carries a pseudo-Riemannian metric of signature $(2n-1, \, 2n + 1)$. Moreover, there is an isometric and transitive action of ${\rm Sp}(1, \B)$ on $S_+$. The Killing vector-fields corresponding to the basis of the 3-dimensional Lie algebra $\mf{sp}(1,\B)$ and the restricted metric on $S_+$ determine a split 3-Sasakian structure on $S_+$.

\subsection{Level-set of the hypersymplectic potential}
Suppose now that $M$ is a hypersymplectic manifold with a free action of ${\rm Sp}(1,\B)$, such that the obstruction $\Xx_2$ vanishes. Then we have a canonically defined hypersymplectic potential on $M$. Consider the level-set $\Ss:=\rho_0^{-1}\left(\frac{1}{2} \right)$. Then, $\Ss$ is ${\rm Sp}(1,\B)$-invariant and the Killing vector fields
\eqst{
K^M_{\xi_1} \, = \, I\, \euler, \, \, \, K^M_{\xi_2} \, = \, -S\, \euler, \, \, \, K^M_{\xi_3} \, = \, -T\, \euler
}
can be thought of as vector fields on $\Ss$. We denote their restriction to $\Ss$ by $K^{\Ss}_{\xi_i}$. Let $g_{\sst \Ss} := \iota^{\ast} g_{\sst M}$ denote the restriction of the hypersymplectic metric to the hypersurface $\Ss$. Define 
\eqst{
\widehat{\eta}_i(Y) \, = \,\tau_i \, g_{\sst \Ss}\,(K^{\Ss}_{\xi_i}, \, Y), \,\,\, \forall Y \in T\Ss, \,\,\, \text{where} \,\,\, \tau_i \, = \, g_{\sst \Ss}\,(K^{\Ss}_{\xi_i}, K^{\Ss}_{\xi_i}).
}
Note that 
\eq{
\label{eq: eta identities}
\widehat{\eta}_1 = \tau_1 \,\iota^{\ast}\pair{\gamma_1, \, \imag }, \,\,\,\,\, \widehat{\eta}_2 = \tau_2 \,\iota^{\ast}\pair{\gamma_1, \, \mathbf{s} }, \,\,\,\,\, \widehat{\eta}_3 = \tau_3 \,\iota^{\ast}\pair{\gamma_1, \, \mathbf{t}}\,\,\,\, \text{and therefore,} \,\,\,\, d\, \widehat{\eta}_i = \tau_i\, \iota^{\ast}\omega_i.
} 
Define the 1-forms
\eqst{
\widehat{\Phi}_1 (Y) = IY + \widehat{\eta}_1(Y) \euler, \,\,\,\,\, \widehat{\Phi}_2 (Y) = SY + \widehat{\eta}_2(Y)  \euler, \,\,\,\,\,\widehat{\Phi}_3 (Y) = TY + \widehat{\eta}_3(Y)  \euler.
}

\begin{thm}
\label{thm: level set of potential is split 3-sasakian}
The manifold $\left(\Ss, \widehat{g_{\sst \Ss}}, \{\,K^{\Ss}_{\xi_i}, \, \widehat{\eta}_i, \, \widehat{\Phi}_i\, \}_{i=1}^3 \right)$ is a split 3-Sasakian manifold. 
\end{thm}

\begin{proof}
Owing to Theorem \ref{thm: split 3-sasakian}, we need only show that $\left(\Ss, \widehat{g_{\sst \Ss}}, \{\,K^{\Ss}_{\xi_i}, \, \widehat{\eta}_i, \,\widehat{\Phi}_i\, \}_{i=1}^3 \right)$ is a metric split 3-contact manifold. For that, it is enough to show that $(g_{\sst \Ss}, \,K^{\Ss}_{\xi_2}, \, \widehat{\eta}_2, \, \widehat{\Phi}_2)$ is a para contact metric structure. From their definitions, it is clear that 
\eqst{
\widehat{\eta}_2\left(K^{\Ss}_{\xi_2} \right) = 1, \,\,\,\, \widehat{\Phi}_2\circ \widehat{\Phi}_2 \, = \, \id_{T\Ss} \, - \, \, \,\widehat{\eta}_2\otimes K^M_{\xi_2}, \,\,\,\,\,\, \widehat{\Phi}_2 \left(K^{\Ss}_{\xi_2} \right) = 0.
}
Therefore we need only show that 
\alst{
& \widehat{g_{\sst \Ss}}(\widehat{\Phi}_2(X), \widehat{\Phi}_2(Y)) =  -\widehat{g_{\sst \Ss}}(X, Y) \, + \, \epsilon_2 \, \widehat{\eta}_2(X)\,\widehat{\eta}_2(Y),\,\,\,\,\, d\widehat{\eta}_2 \,(X,Y) \, = \, \widehat{g_{\sst \Ss}}(\widehat{\Phi}_2(X), \, Y)\,\,\,\,\, X, Y \in T\Ss.
} 

\noindent Let $\eta_i \,(\cdot) = \,g_{\sst M}\left(K^M_{\xi_i}, \, \cdot \right)$. Then $\widehat{\eta_i} = \iota^{\ast} \eta_i$. Observe that for any $X, Y \in T\Ss$, 
\alst{
\widehat{g_{\sst \Ss}}(\widehat{\Phi}_2(X), \widehat{\Phi}_2(Y)) \, 
& = \, g_{\sst M} \left(SX + \eta_2(X) \euler ,\, SY + \eta_2(Y) \euler \right) \\
& = \, -g_{\sst M}\left(X, Y \right)  +  \eta_2(X)\eta_2(Y) \,g_{\sst M}(\euler, \euler) + \eta_2(X) \,g_{\sst M}(SY,  \euler)  + \eta_2(Y)\,g_{\sst M}(SY,  \euler)\\
& = \, -g_{\sst M}\left(X, Y \right)  +  \eta_2(X)\eta_2(Y) \,g_{\sst M}(\euler, \euler) + \eta_2(X) \,g_{\sst M}(Y,  K^M_{\xi_2})  + \eta_2(Y)\,g_{\sst M}(Y,  K^M_{\xi_2})\\
& = \, -g_{\sst M}\left(X, Y \right) \, + \, 2\rho_0 \cdot \,\eta_2(X)\eta_2(Y)\, + \, \tau_2\, \eta_2(X)\eta_2(Y) \, + \,\tau_2\, \eta_2(X)\eta_2(Y)\\
& = \, -g_{\sst M}(X, Y) \, + \,(2 - 2\rho_0) \, \tau_2 \, \eta_2(X)\,\eta_2(Y) \,\,\,\,\,\,\,\,\,\,\,\,\,\,\,\,\, \text{(since $\tau_2 = -1$)}\\
& = -\widehat{g_{\sst \Ss}}(X, Y) \, + \, \tau_2 \, \widehat{\eta}_2(X)\,\widehat{\eta}_2(Y)\\
& = \varepsilon_2 \left(\,-\widehat{g_{\sst \Ss}}(X, Y) \, + \, \tau_2\, \widehat{\eta}_2(X)\,\widehat{\eta}_2(Y)\, \right)\,\,\,\,\,\,\,\,\,\,\,\,\,\,\,\,\, \text{(since $\varepsilon_2 = 1$)}.
}

We now check the second condition
\alst{
\widehat{g_{\sst \Ss}}(\widehat{\Phi}_2(X), \, Y) \,
& = \,g_{\sst M}\left( SX + \widehat{\eta}_2(X) \euler ,\, Y\right)\\
& = \, g_{\sst M}\left(SX, \, Y \right) \, + \, \widehat{\eta_2}(X)\,g_{\sst M}\left(\euler, Y\right)\\
& = g_{\sst M}\left(SX, \, Y \right) \, = \, \omega_2(X, Y) \, = \, \iota^{\ast}\omega_2(X, Y),
}
for all $X, Y \in T\Ss$. It follows from eq. \eqref{eq: identity of omega} and eq.\eqref{eq: eta identities} and the fact that external derivative commutes with the pull-back that $d\widehat{\eta_2} =\tau_2 \,\iota^{\ast}\omega_2$. This shows that $d\eta_2(X,Y) = \tau_2\,\widehat{g_{\sst \Ss}}(\widehat{\Phi}_2(X), \, Y)$. Thus $(\widehat{g_{\sst \Ss}}, \,\xi_2, \, \widehat{\eta}_2, \, \widehat{\Phi}_2)$ defines a para contact metric structure on $S$. 

Similar arguments show that $(\widehat{g_{\sst \Ss}}, \,\xi_1, \, \widehat{\eta}_1, \, \widehat{\Phi}_1)$ and $(\widehat{g_{\sst \Ss}}, \,\xi_3, \, \widehat{\eta}_3, \, \widehat{\Phi}_3)$ define a pseudo and a para contact metric structure respectively. Moreover, the vector fields $(K^{\Ss}_{\xi_1}, K^{\Ss}_{\xi_2}, K^{\Ss}_{\xi_3})$ clearly satisfy the split quaternionic relations \eqref{eq: split quaternionic relations}. The claim thus follows from the statement of Theorem \ref{thm: split 3-sasakian}.
\end{proof}

Note that since the metrics on the level-sets $\Ss_c := \rho_0^{-1}(c)$ are homothetic, a split 3-Sasakian structure can be defined on every level-set.

\begin{note}
\label{note: note on hypersympl. manifolds as metric cones}
Analogous to the hyperK\"ahler case, it can be easily seen that a metric cone over any split 3-Sasakian manifold is a hypersymplectic manifold, with a free, permuting action of ${\rm Sp}(1,\B)$ and a hypersymplectic potential. In particular, the obstruction $\rho_2$ vanishes. This can be considered as a characterizing property of such hypersymplectic manifolds. In other words, if $\rho_2$ vanishes, then the hypersymplectic manifold in consideration, can be written as a metric cone over a split 3-Sasakian manifold, as constructed above.
\end{note}

\subsection{Para quaternionic K\"ahler manifolds}

A \emph{para-quaternionic K\"ahler manifold} is a, pseudo-Riemannian manifold of dimension $4n$, whose holonomy is contained inside the group 
\eqst{
{\rm Sp}(n, \B)\cdot {\rm Sp}(1, \B) \, = \, {\rm Sp}(n, \B)\times_{\pm 1} {\rm Sp}(1, \B).
}
Equivalently, we say that a manifold $N$ is \emph{almost para-quaternionic K\"ahler manifold} if there exists a sub-bundle $\Ii'\subset \End(TN)$ which is locally spanned by a triple $(I, S, T)$ satisfying the split quaternionic relations \eqref{eq: split quaternionic relations}. For $n >1$, the requirement that the holonomy of $N$ be contained inside ${\rm Sp}(n, \B)\cdot {\rm Sp}(1, \B)$ is equivalent to asking the sub-bundle $\Ii'$ being preserved by the Levi-Civita connection. If $n \geq 3$, then, this is equivalent to showing that the globally defined 4-form
\eqst{
\Omega \, = \, \omega_1 \wedge \omega_1 \, - \, \omega_2 \wedge \omega_2 \, - \, \omega_3 \wedge \omega_3
}
is closed. For $n=1$, we additionally require that the manifold be self-dual and Einstein.

\begin{thm}\cite{gml01}
Any para-quaternionic K\"ahler manifold $(N, g_{\sst N}, \widehat{\Omega})$ is Einstein, provided that the dimension of $N$ is greater than 4.
\end{thm}

The representation theoretic argument by S. Salamon \cite{sal82} can also be adapted to the pseudo-Riemannian setting to prove the above theorem.

A wide range of examples of para-quaternionic K\"ahler manifolds can be constructed by adapting LeBrun's construction of quaternionic K\"ahler manifolds \cite{lebrun89}, to the pseudo-Riemannian case \cite{djs05}. Starting with a real analytic manifold $F$ of dimension $2n + 1$, endowed with an indefinite metric, it is possible to construct a para-quaternionic manifold of dimension $4n$. Different manifolds which are conformal to $F$ give rise to distinct para-quaternionic K\"ahler manifolds. We can thus construct a wide variety of para-quaternionic K\"ahler manifolds of dimension greater than four.

Looking at Berger's list \cite{berg57, djs05}, one can also construct symmetric para-quaternionic K\"ahler manifolds of the type $G/H$, where $G$ is semi-simple. Symmetric para-quaternionic K\"ahler manifolds have been completely classified by D. Alekseevsky and V. Cort\'es \cite{ac05}.

In the hyperK\"ahler case, the quotient of the 3-Sasakian manifold (a level-set of the hyperK\"ahler potential) by the group ${\rm Sp}(1)$ produces a quaternionic K\"ahler manifold of positive scalar curvature. This, however, cannot be directly carried over to the hypersymplectic situation as the group ${\rm SU}(1,1)$ is non-compact and therefore the quotient may not even be Hausdorff. However, if the ${\rm Sp}(1,\B)$-action \textbf{is proper}, then, we show that the quotient of the split 3-Sasakian manifold by ${\rm Sp}(1,\B)$ is a para-quaternionic K\"ahler manifold. Henceforth, we assume that the ${\rm Sp}(1,\B)$-action on $M$ proper.

Let $N = \Ss/{\rm Sp}(1,\B)$ be the quotient of the split 3-Sasakian manifold and consider the diagram,

\begin{wrapfigure}{r}{0.45\textwidth}
\vspace{-4mm}
\hspace{9mm}
\begin{tikzpicture}[->, node distance=3cm, auto, shorten >=1pt, every edge/.style={font=\footnotesize, draw}, fill=black]
         \node (Q)     {$\Ss$};
         \node (M)    [right of=Q]    {$M$};
         \node (X)    [below of=Q]    {$N$};

         \draw                     (Q) -- node [left, midway] {$\pi$} (X);        
         \draw  [right hook-latex] (Q) -- node [above, midway] {$\iota$} (M);
                  
\end{tikzpicture}
\vspace{3mm}
\label{diag: quat-kah and sasakian}
\end{wrapfigure}
\noindent where $\iota$ is the pseudo-Riemannian embedding and the map $\pi$ is the pseudo-Riemannian principal submersion. The normal bundle $\Nn$ is the 1-dimensional vector bundle $\iota^{\ast}\text{span}\{\euler\} \subset \iota^{\ast}TM$. The pull-back bundle $\iota^{\ast}TM$ splits into the direct sum $\iota^{\ast}TM = \Nn \oplus \,T\Ss $,

\noindent The pull-back metric on $\Ss$ is of signature $(2n-1, 2n)$. Further, $T\Ss$ splits into a direct sum $\Vv \oplus \Hh$, where
\eqst{
\Vv \, = \, \text{Span}\left\lbrace K^M_{\xi_1}, K^M_{\xi_2}, K^M_{\xi_3} \right\rbrace \,\,\,\, \text{and} \,\,\,\, \Hh \, = \,\cap_{i = 1}^3 \ker \widehat{\eta}_i.
}
In conclusion, the pullback-bundle splits as
\eqst{
\iota^{\ast} TM \, = \, \Nn \, \oplus \, \Vv \, \oplus \, \Hh.
}
We will call $\Hh$ the ``horizontal bundle" of $T\Ss$. Let $\beta_i(\cdot, \cdot) = \widehat{g}_{\sst \Ss}(\widehat{\Phi}_i(\cdot), \cdot)$ and $\theta_i := \beta_i + \epsilon_{ijk} \, \eta_j \wedge \eta_k$. Define the 4-form
\eqst{
\widehat{\Omega}\, = \, \,\theta_1 \wedge \theta_1 \, - \, \theta_2 \wedge \theta_2 \, - \, \theta_3 \wedge \theta_3 \,\,\,\, i=1,2,3.
}
Then, $\widehat{\Omega}$ is an ${\rm Sp}(1, \B)$-invariant, horizontal 4-form and therefore it descends to a 4-form $\Omega_N$ on $N$ with $\pi^{\ast}\Omega_N = \widehat{\Omega}$.

Observe that since $\rho_2 = 0$, we have that 
\eqst{
K^{\Ss}_{\xi_i} \, = -\varepsilon_i \, \I_{\xi_i}\euler, \,\,\,\,\,\,\,\,\, \I_{\xi_i}K^{\Ss}_{\xi_j} \, = \, \tau^{ijk}\, K^{\Ss}_{\xi_k} \, + \, \delta_{ij}\, \varepsilon_i\, \euler.
}
where $\tau^{ijk}$ denotes the sign of the permutation $(i,j,k)$. It follows that $\Vv \oplus \Nn$ is invariant under $I, S, T$ and therefore, $\Hh$ is invariant under $I,S,T$. We thus get an ${\rm Sp}(1,\B)$-invariant almost para-quaternionic K\"ahler structure on $\Hh$, which descends to $N$. The 4-form associated to the almost para-quaternionic K\"ahler structure is $\Omega_N$. In order to show that the structure is para-quaternionic K\"ahler, we need to show that the quotient metric $g_{\sst N}$ has holonomy group contained in ${\rm Sp}(n, \B)\cdot {\rm Sp}(1, \B)$. Or equivalently,
\eqst{
\nabla^N\, \Omega_N \, = \, 0,
}
where $\nabla^N$ is the Levi-Civita connection of the metric $g_{\sst N}$. Observe that for any $W, X, Y, Z \in \Hh$, 
\eqst{
\widehat{\Omega}\,(W, X, Y, Z) = (\iota^{\ast}\omega_1\wedge\iota^{\ast}\omega_1 \, - \, \iota^{\ast}\omega_2\wedge\iota^{\ast}\omega_2 \, - \, \iota^{\ast}\omega_3\wedge\iota^{\ast}\omega_3)\, (W, X, Y, Z).
}
The Levi-Civita connection on $M$ induces a connection $\nabla^{\Ss}$ on $\Ss$, which is precisely the Levi-Civita connection of the pull-back metric $\iota^{\ast}g_{\sst M} = \widehat{g}_{\sst \Ss}$. It follows that $\nabla^{\Ss}\widehat{\Omega} = 0$. From the fact that $\Ss \rightarrow N$ is a pseudo-Riemannian submersion and the standard computation using O' Neil's formula (\cite{gl88}, Thm. 3.1), we conclude that the holonomy of the quotient metric is a sub-group of ${\rm Sp}(n, \B)\cdot {\rm Sp}(1, \B)$. Thus, $N$ is a para quaternionic K\"ahler manifold. Note that the signature of $g_{\sst N}$ is $(2n-2, 2n-2)$. To sum-up

\begin{thm}
\label{thm: pqk manifold as quotient of split 3-sasakian}
Suppose that $M$ is a hypersymplectic manifold of dimension $4n$ with a free and proper permuting action of ${\rm Sp}(1,\B)$. Assume that the obstruction $\rho_2 = 0$. Then, the quotient of any level set $\rho_0^{-1}(c)/ {\rm Sp}(1, \B)$ is a para-quaternionic K\"ahler manifold of dimension $4n-4$, endowed with a metric of signature $(2n-2, 2n-2)$.
\end{thm}

The converse of the above statement is also true. Namely, if the quotient $\rho_0^{-1}(c)/ {\rm Sp}(1, \B)$ is para-quaternionic K\"ahler, then, $\rho_2 = c$. This follows directly from the arguments in proof of Theorem 2.15 of \cite{bgm93} and Lemmas \ref{lem: identity of omega}, \ref{lem: identity of gamma_1} and \ref{lem: euler vf}. 

\subsection*{Swann bundles on para-quaternionic K\"ahler manifolds}
Going in the the other direction, given a para-quaternionic K\"ahler manifold, consider its reduced ${\rm Sp}(n,\B)\cdot {\rm Sp}(1, \B)$-frame bundle $F$. Then, $\scr{S}(N) = F/{\rm Sp}(1,\B)$ is a principal ${\rm SO}^+(1,2)$ bundle. Let $\widetilde{\B}$ denote the space of all the zero-divisors in $\B$ and let $\B^{\ast} := (\B \setminus\{ 0 \}) \setminus \widetilde{\B}$. %

The permuting ${\rm Sp}(1, \B)$-action on $\B$, given by $(q, h) \mapsto h\overline{q}$, $h \in \B$, $q \in {\rm Sp}(1, \B)$, descends to an action of ${\rm SO}^+(1,2)$ on $\B^{\ast}/\Z_2$. Note that the action is transitive. There is yet another action of ${\rm Sp}(1, \B)$ on $\B^{\ast}$, which is given by left multiplication and which commutes with the first one. It therefore descends to an action of ${\rm SO}^+(1,2)$ on $\B^{\ast}/\Z_2$.

Consider the principal bundle
\eqst{
\swann_+ \, := \, \scr{S}(N) \times_{{\rm SO}^+(1,2)}(\B^{\ast}/\Z_2) \longrightarrow N.
}

The total space $\swann_+$ of the bundle, is a hypersymplectic manifold, with the induced (permuting) action of ${\rm SO}^+(1,2)$ on the fibres. This is the \emph{positive Swann bundle} constructed by Dancer, J{\o}rgensen and Swann \cite{djs05}. Note that since $\B^{\ast} \cong \rm{Sp}(1,\B) \times (\R \setminus \{0\})$, the fibre at each point
\eqst{
\B^{\ast}/\Z_2 \, \cong \, \rm{SO}^+(1,2) \times \R_{>0}.
}
Here, we think of $\Z_2$ as the normal sub-group generated by the element ${(-1,-1)} \in \rm{Sp}(1,\B) \times (\R \setminus \{0\})$. The isomorphism is seen quite easily by considering the adjoint action of $\B^{\ast}$ on $\B$:
\eqst{
(q, r) \cdot h \, \longmapsto \, (rq)\, h \, (\overline{rq}), \,\,\,\, \text{where} \,\,\, q \in {\rm Sp}(1, \B), \,\,\, r \in \R\setminus\{0\} \,\,\, \text{and} \,\,\, h \in \B.
}
Therefore, the total space of the Swann bundle can alternatively be written as a metric cone
\eqst{
\swann_+ \, = \, \scr{S}(N) \times \R_{>0}, \,\,\,\,\,\,\, g_{\sst \swann} \, = \, dr^2 \, + \, r^2 \, \left(g_{\sst N} + g_{\sst {\rm SO}^+(1,2)} \right).
}
This shows that $\scr{S}(N)$ is a split 3-Sasakian manifold. Note that the hypersymplectic potential on $\swann_+$ is just $\rho_0(s, r) = \frac{1}{2}\, r^2$ and the Euler vector field $\euler = r \,\partial/ \partial r$.

Using Le Brun's construction \cite{lebrun89}, one can construct a para-quaternionic K\"ahler manifold of dimension $4k$ from a real-analytic, pseudo-Riemannian manifold of dimension $k+1$ \cite{djs05}. Moreover, different manifolds, conformal to the real-analytic manifold give rise to distinct para-quaternionic K\"ahler manifold. In this way, we have a plethora of examples of para-quaternionic K\"ahler manifolds and therefore also of hypersymplectic manifolds with a permuting ${\rm SO}^+(1,2)$-action.


\section{Examples}

\subsection{Commutativity of constructions and split quaternionic modules}
\label{ex: split quaternionic modules}

Many non-trivial examples of hyperK\"ahler manifolds with permuting ${\rm Sp}(1)$-action are obtained as reductions of the flat-space $\HH^{n}$. Analogously, we use the Marsden-Weinstein construction for constructing non-trivial examples of hypersymplectic manifolds, carrying a permuting ${\rm Sp}(1,\B)$-permuting action, starting with the flat-space $\B^n$. 

Suppose that $(M, g_{\sst M}, I, S, T)$ is a hypersymplectic manifold, with a free and isometric action of a compact Lie group $G$ that preserves the hypersymplectic structure. Let $\mu: M \rightarrow \mf{g}^{\ast}\otimes\mf{sp}(1,\B)^{\ast}$ denote a hypersymplectic moment map and assume that $0$ is a regular value of the moment map. Let $\mf{O} = \{\, K^M_{\xi}\, \vert \, \xi\in \mf{g}\, \}$. Suppose that the following holds:
\begin{enumerate}
\item $G$ acts freely on $\mu^{-1}(0)$
\item Rank $d\mu$  = 3 dim $\mf{g}$ at each point of $M$
\item $\mf{O} \cap \mf{O}^{\perp} = \{0 \}$.
\end{enumerate}

The above conditions are referred to as conditions (F), (S) and (D) respectively in \cite{djs05}.

\begin{thm}[Hypersymplectic reduction, \cite{hit90}]
The quotient $M' = \mu^{-1}(0)/G$ is a smooth hypersymplectic manifold, with respect to the induced quotient hypersymplectic structure $(I', S', T')$ and the induced quotient metric $g'_{\sst M'}$.
\end{thm}

Conditions (1) and (2) guarantee a smooth quotient manifold while (3) says that the hypersymplectic structure is non-degenerate. In contrast with the hyperK\"ahler case, (2) does not follow from (1) automatically, owing to the fact that the metric on $M$ is indefinite. In most known examples, (3) is usually never guaranteed. Consequently, the quotients are hypersymplectic manifolds, away from some singular locus. For example, suppose $G = \rm{U}(1)$. Then, (3) is violated precisely along the locus of all those points $p$, where $g_{\sst M}(K^M_{\imag}\vert_p,K^M_{\imag}\vert_p) = 0$.


Let $(N, g_{\sst N}, \Omega_N)$ be a para-quaternionic K\"ahler manifold. Let $(I,S,T)$ denote the local basis for the paraquaternionic K\"ahler structure and $\eta_I, \eta_S, \eta_T$ denote the corresponding local para-K\"ahler/ pseudo-K\"ahler 2-forms. Locally, for a Killing vector field $X$, we can define the 1-form 
\eqst{
\Theta (X) \, = \, \imag \,\eta_I(X, \cdot) \, + \, \mbf{s}\, \eta_S(X, \cdot) \, + \, \mbf{t}\, \eta_T(X, \cdot).
}
The form is independent of the the choice of the local basis $(I,S,T)$ and is therefore globally defined.

\begin{thm}[Para-quaternionic K\"ahler reduction, \cite{vuk03} (Thm. 5.2)]
\label{thm: para qk reduction}
Let $G$ be a Lie group acting freely and isometrically on a para-quaternionic K\"ahler manifold $(N, g_{\sst N}, \Omega_N)$. Assume that the group action preserves $\Omega_N$. Then, there exists a unique map $\mu: N \rightarrow \mf{g}^{\ast}\otimes \mf{sp}(1, \B)$ such that $d\mu = \iota_{\mf{g}}\, \Theta$.

Suppose that $\mu^{-1}(0)$ is a smooth submanifold of $N$, on which $G$ acts freely and properly, so that $N' = \mu^{-1}(0)/G$ is a smooth pseudo-Riemannian submersion. Then, $(N', g'_{\sst N'}, \Omega'_{N'})$ is again a para-quaternionic K\"ahler manifold, with respect to the induced para-quaternionic K\"ahler structure $\Omega'_{N'}$ and the induced metric $g'_{\sst N'}$
\end{thm}

With the above two theorems at hand, we can directly adapt Swann's arguments in \cite{swann91} to the pseudo-Riemannian setting to show that the quotient construction commutes with reduction. Namely,

\begin{thm}
\label{thm: commutativity of quotients}
Let $(N, g_{\sst N}, \Omega_N)$ be a para-quaternionic K\"ahler manifold. Suppose that a Lie group $G$ acts isometrically,freely and properly, preserving the para-quaternionic K\"ahler structure. Then, $G$ induces an isometric action on $\swann$, which preserves the hypersymplectic structure on $\swann$. Moreover, the hypersymplectic quotient of $\swann$ by the $G$ action is the total space of the Swann bundle over the para-quaternionic K\"ahler quotient of $N$ by $G$.
\end{thm}

\smallskip

Let us now consider the split quaternionic module $\B^{n+1}$. Let $\widetilde{\B}$ denote the sub-space of all the null-vectors in $\B^{n+1}$ and consider the sub-space of all the space-like vectors (positive norm) $\left(\B^{n+1} \right)^{\ast} := (\B^{n+1} \setminus \{0\})\setminus \widetilde{\B}$. Then $\left(\B^{n+1} \right)^{\ast}$ is a union of two disjoint spaces of space-like and time-like vectors. Let $\left(\B^{n+1} \right)^{\ast}_+$ denote the sub-space of space-like vectors. We will show that this is the total space of a Swann bundle over a para-quaternionic K\"ahler manifold. First, observe that $\left(\B^{n+1} \right)^{\ast}_+$ is a hypersymplectic manifold, equipped with a free and proper action of ${\rm Sp}(1,\B)$, as described in Subsec. \ref{subsec: modules over split quaternions}. Moreover, it also carries a homothetic action of $\R^+$ given by $(r,h) \mapsto r\cdot h$. Clearly, we see that the obstruction $\rho_2$ vanishes and the hypersymplectic potential is given by $\rho_0(h) = \frac{1}{2}\norm{h}^2$.

\noindent Consider the positive sphere 
\eqst{
S_+ \, := \, \rho_0^{-1}\left(\frac{1}{2} \right) \, = \, \{~ q \in (\B^{n+1})^{\ast}_+ ~ | ~ \norm{q}^2 \, = \, 1 ~\} \, \cong \, \frac{{\rm Sp}(n+1, \B)}{{\rm Sp}(n)}.
} 
The sphere carries a metric of signature $(2n-1, 2n+1)$. Then, clearly, $\left(\B^{n+1} \right)^{\ast}_+$ is topologically a metric cone over $S_+$ and so $\left(\B^{n+1} \right)^{\ast}_+ \cong S_+ \times \R_{>0}$. The hypersymplectic potential is just $\rho_0(s, r) = \frac{1}{2}\, r^2$. Therefore, $S_+$ is a split 3-Sasakian manifold. The ${\rm Sp}(1, \B)$-action on $(\B^{n+1})^{\ast}_+$ induces a free, proper and isometric action of ${\rm Sp}(1, \B)$ on $S_+$. By Theorem \ref{thm: pqk manifold as quotient of split 3-sasakian}, the quotient $S_+/ {\rm Sp}(1, \B)$ is a para-quaternionic K\"ahler manifold, which is nothing but the para-quaternionic projective space
\eqst{
\B\P^n \, = \, \frac{{\rm Sp}(n+1, \B)}{{\rm Sp}(n) \times {\rm Sp}(1, \B)}.
}
It follows that $\left(\B^{n+1} \right)^{\ast}_+$ is the total space of the Swann bundle over $\B\P^n$, i.e, $\Uu(\B\P^n)$. This is the \emph{positive Swann bundle} described in \cite{djs05}. Split quaternionic projective spaces have been studied by Bla{\v z}i{\' c} \cite{bla96} and Wolf \cite{wolf}. 

Consider an action of a Lie group $G \subset {\rm Sp}(n+1, \B)$ on $\B^{n+1}$, that commutes with the permuting ${\rm Sp}(1,\B)$-action by right conjugate multiplication. The induced action on $(\B^{n+1})^{\ast}_+$ preserves the hypersymplectic structure and therefore also the three symplectic forms. So, there exist three moment maps, which we combine into a single $G \times {\rm Sp}(1,\B)$-equivariant map
\eqst{
\mu: (\B^{n+1})^{\ast}_+ \longrightarrow \mf{sp}(1,\B) \otimes \mf{g}^{\ast}, \,\,\,\,\,\,\,\, \mu \, = \, \imag\,\mu_1 \, + \, \mbf{s}\,\mu_2 \, + \, \mbf{t}\, \mu_3.
}
Suppose that $0$ is a regular value of $\mu$ and $G$ acts freely and properly on the zero level-set of the moment map. Additionally, assume that the metric, restricted to the group orbits in $\mu^{-1}(0)$ is non-degenerate. Hitchin's work \cite{hit90} now guarantees that the quotient $\widetilde{\Mm} = \mu^{-1}(0)/G$ is a hypersymplectic manifold. The permuting action of ${\rm Sp}(1,\B) \times \R^+$ on $(\B^{n+1})^{\ast}_+$ commutes with the action of $G$ (therefore preserving the zero-level-set of $\mu$) and hence descends to the quotient $\widetilde{\Mm}$. In particular the obstruction $\rho_2$ vanishes; i.e., $\widetilde{\Mm} = \swann$ for some para-quaternionic K\"ahler manifold $N$. Since ${\rm Sp}(1,\B) \times \R^+$-action commutes with that of $G$, the latter descends to a para-quaternionic K\"ahler action on $\B\P^n$. By Theorem \ref{thm: commutativity of quotients}, it follows that $N$ is the para-quaternionic K\"ahler reduction of $\B\P^n$ by $G$.

When $G$ is a compact subgroup of $\T^{n+1}$, Dancer and Swann \cite{ds07} show that the hypersymplectic reduction of $\B^{n+1}$ by $G$ is, a hypersymplectic manifold, with a non-trivial de-generacy locus - the set of all the points where the metric is de-generate along the orbits of the $G$-action. For example, when $G = {\rm U}(1)$, this is precisely the set of points where $g_{\sst M} (K^M, K^M) = 0$, where $K^M$ is the fundamental vector field due to ${\rm U}(1)$-action. The ${\rm U}(1)$-action descends to an action on $(\B^{n+1})^{\ast}_+$ and the moment map is just the restriction of $\mu$. From the discussion above, it follows that the reduced manifold is a smooth hypersymplectic manifold, which is the total space of a Swann bundle over the para-quaternionic K\"ahler reduction of $\B\P^n$ by ${\rm U}(1)$.

\subsection{Moduli spaces of Nahm-Schmid equations}
We give another set of examples of hypersymplectic manifolds, carrying a permuting ${\rm SU}(1,1)$-action, for which the obstruction $\rho_2$ vanishes. Namely, we consider the moduli space of Nahm-Schmid equations. The equations can be interpreted as the zero-level set of infinite-dimensional hypersymplectic moment map for the action of gauge group on the configuration space. Consequently, the moduli space of solutions is an infinite-dimensional hypersymplectic reduction. We refer to \cite{brr18} for more details. 

\medskip

\paragraph{\textbf{Nahm-Schmid equations:}}
Nahm-Schmid equations are pseudo-Riemannian analogues of Nahm's equations and arise as dimensional reduction of Yang-Mills equations on $\R^{2,2}$. Let $G$ be a compact Lie group and $\mf{g}$ denote its Lie algebra. Let $T_i: \R \rightarrow \mf{g}$ be $C^1$-differentiable maps for $i = 0,1,2,3$. The \emph{Nahm-Schmid} equations for $\{T_i\}_{i=0}^3$ is a system of equations, satisfying
\al{
\label{eq: nahm-schmid eqns.}
\dot{T}_1 \, + \, [T_0, \, T_1] \, = \, -[T_2, \, T_3], \,\,\,\,
\dot{T}_2 \, + \, [T_0, \, T_2] \, = \, [T_3, \, T_1], \,\,\,\, \dot{T}_3 \, + \, [T_0, \, T_3] \, = \, [T_1, \, T_2].
}
The equations are invariant under the action of the \emph{gauge group} $\Gg := C^2(\R, G)$. Moreover, there exists a gauge transformation such that $T_0 = 0$ and therefore we may assume, without loss of generality, that $T_0 = 0$ in the above equations. Using this, we get the \emph{reduced Nahm-Schmid equations}:
\al{
\label{eq: reduced nahm-schmid eqns.}
\dot{T}_1 \, = \, -[T_2, \, T_3], \,\,\,\,
\dot{T}_2 \, = \, [T_3, \, T_1], \,\,\,\, 
\dot{T}_3 \, = \, [T_1, \, T_2].
}

\noindent Let $\norm{\cdot}$ denote the Ad-invariant inner product on the Lie algebra $\mf{g}$.

\begin{prop}[Prop. 2.2, \cite{brr18}]
\label{prop: NS conserved quantity}

Let $(T_1, T_2, T_3)$ be a solution to \eqref{eq: reduced nahm-schmid eqns.}. Then, 
\eq{
\label{eq: NS conserved quantity}
2\norm{T_1}^2 \, + \, \norm{T_2}^2 \, + \, \norm{T_3}^2 \, = \, \text{const.}
}
Consequently, the solutions exist for all times.
\end{prop}

\subsection*{Nahm-Schmid equations on [0,1]}
Let $G$ be a compact Lie group and $\mf{g}$ denote its Lie algebra. We will consider the solutions $\{ T_i \}_{i=0}^{3}$ defined on the interval $[0,1]$. 

Let
\eqst{
\Aa \, = \, \lbrace T_i: [0,1] \longrightarrow \mf{g} \,\, \text{for} \,\, i = 0,1,2,3 \,\, \vert \,\, T_i \,\, \text{is differentiable of class} \,\, C^1 \rbrace.
}

We can identify $\Aa \cong C^1([0,1], \mf{g})\otimes \B$ by mapping:
\eqst{
(T_0, T_1, T_2, T_3) \longmapsto \Tt \, := \, T_0 \, + \, \imag \, T_1 \, + \, \mbf{s} \, T_2 \, + \, \mbf{t} \, T_3.
}
The tangent space to $\Aa$ at a point $\Tt$ is given by $T_{\Tt} \Aa = C^1([0,1], \mf{g})\otimes \B$. There is a split-quaternionic structure on $\Aa$, which is induced by the split-quaternionic structure on $\B$, given by:
\eqst{
I(X) \, = \, X \cdot \overline{\imag}, \,\,\,\,\, S(X) \, = \, X\cdot \mbf{s}, \,\,\,\,\, T(X) \, = \, X \cdot \mbf{t}, \,\,\,\,\,\, \text{for} \,\,\, X \in T_{\Tt} \Aa.
}
The space $\Aa$ is equipped with the indefinite, neutral signature metric, induced from the metric on $\B \cong \R^{2,2}$ and an ad-invariant metric $\pair{\cdot, \cdot}$ on $G$
\eqst{
g_{\sst \Aa}(X, Y) \, = \, \int_{[0,1]}\, \pair{X(t) \cdot \overline{Y(t)}} \, = \, \int_{[0,1]}\, \sum_{i=0}^3 \, a_{ii} \pair{X_i(t), \, Y_i(t)},, \,\,\,\,\, X, Y \in T\Aa,
}

%
\noindent where $a_{00} = a_{11} = 1$ and $a_{22} = a_{33} = -1$. The metric is compatible with the split-quaternionic structure $I, S, T$. This endows $\Aa$ with a structure of a flat hypersymplectic manifold. The hypersymplectic structure is preserved by the action of the gauge group $\Gg:=C^2([0,1], G)$. Let $\Gg_{00}$ be the normal subgroup of $\Gg$, given by
\eqst{
\Gg_{00} \, := \, \{\, g \in \Gg \,\, \vert \,\, g(0) = \mathds{1}_G = g(1) \,\, \}.
}
Then, the infinite-dimensional moment maps for the action of $\Gg_{00}$ on $\Aa$ are given by
\alst{
\mu_I (\Tt) \, = \, \dot{T_1} \, + \, [T_0, T_1] \, + \, [T_2, T_3] \\
\mu_S (\Tt) \, = \, \dot{T_2} \, + \, [T_0, T_2] \, - \, [T_3, T_1] \\
\mu_T (\Tt) \, = \, \dot{T_3} \, + \, [T_0, T_3] \, - \, [T_1, T_2].
}
Writing $\mu_{\sst \Aa} = \mu_{I}\, \imag + \mu_{S}\, \mbf{s} + \mu_{T}\, \mbf{t}$, we see that the solutions to the Nahm-Schmid equations can be interpreted as the zero level-set of the hypersymplectic moment map $\mu_{\Aa}$.

The space $\Aa$ also carries a permuting action of the group ${\rm Sp}(1,\B)$, which is induced by a permuting action of ${\rm Sp}(1,\B)$ on $\B$. More precisely, the action is given by 
\eqst{
(q, \Tt) \longmapsto q \cdot \Tt \cdot \overline{q}, \,\,\,\,\,\, q \in {\rm Sp}(1,\B).
}
Since the action is induced by the permuting action on $\B$, it is isometric, free and proper. Moreover, the action commutes with that of $\Gg_{00}$, thus preserving the zero-level set of $\mu_{\Aa}$. In essence, equations \eqref{eq: nahm-schmid eqns.} are invariant under the action of ${\rm Sp}(1,\B)$. The induced action on the space of product and complex structures $\Lambda_{\scr{H}}$ on $\Aa$ is, pointwise, just the standard action of ${\rm SO}^+(1,2)$ on the pseudo-sphere $\scr{H}$. Namely, suppose that $F$ denotes either of $I, S$ ot $T$, and $\mbf{f}$ correspondingly denotes either $\overline{\imag}, \mbf{s}$ or $\mbf{t}$. Then, the induced action is $(q, F) \mapsto q_{\ast} \, F \, \overline{q_{\ast}}$. Indeed,
\eqst{
(q_{\ast}\,  F \, \overline{q}_{\ast}) \, (q_{\ast}\, X) \, = \, (q \cdot X \cdot \overline{q}) \cdot (q \cdot \mbf{f} \cdot \overline{q}) \, = \,  q \cdot (X \cdot \mbf{f}) \cdot \overline{q} \, = \, q_{\ast} (F(X)), \,\,\,\,\, q \in {\rm Sp}(1,\B).
}


\begin{thm}[\cite{brr18}]
The gauge group $\Gg_{00}$ acts freely and properly on $\Aa$ and therefore the moduli space 
\eq{
\label{eq: moduli space}
\Mm \, = \, \mu^{-1}(0)/\Gg_{00}
}
is a smooth Banach manifold, diffeomorphic to $G \times \mf{g} \times \mf{g} \times \mf{g} \cong G \times (\mf{g} \otimes \mf{Im}(\B))$. 
\end{thm}

\begin{note}
The moduli space $\Mm$, although a smooth manifold, does not carry a smooth hypersymplectic structure. There exists a \emph{degeneracy locus}, which is precisely the locus of points $\Tt$ where the metric, when restricted to the tangent space of the $\Gg_{00}$-orbit through $\Tt$, is degenerate. Outside of this degeneracy locus, $\Mm$ carries a smooth hypersymplectic structure. Note that the ${\rm Sp}(1,\B)$-action and the action of $\Gg_{00}$ on $\Aa$ preserves the metric and therefore also the degeneracy locus. 
\end{note}

\begin{prop}[\cite{brr18}]
The degeneracy locus $\Dd$ is in one-to-one correspondance with those solutions $\Tt = (T_0, T_1, T_2, T_3)$ to Nahm-Schmid equations \eqref{eq: nahm-schmid eqns.}, for which there exists a solution to the following ODE
\eq{
\label{eq: degen locus criterion}
\xi \longmapsto \frac{d^2\xi}{dt^2} \, + \, [T_0, \dot{\xi}] \, + \, [\dot{T_0}, \xi] \, + \, \sum_{i=0}^{3}\, a_{ii}\, [T_i\,,[T_i, \xi]\,] \, = \, 0, \,\,\,\,\, \xi(0) \, = \, 0 \, = \, \xi(1).
}
\end{prop}

\subsection*{Homothetic action of \texorpdfstring{$\R^+$}~ and degeneracy locus} We define a homothetic action of $\R^+$ on $\Aa$ as follows: let $a \in \R$ and consider the map
\alst{
\phi_a: C^{1}([0,1], \,\mf{g}) \longrightarrow C^{1}([0,a^{-1}], \,\mf{g}), \,\,\,\,\,\, T_i(t) \longmapsto a\, T_i\,(at).
}
It is easy to see that the equations \eqref{eq: nahm-schmid eqns.} are invariant under $\phi_a$. As a result, $\phi_a$ maps solutions on $[0,1]$, to solutions defined on $[0,a^{-1}]$. The homomorphism between the gauge groups
\eqst{
\eta_a: \Gg \longrightarrow \Gg_{\sst a^{-1}}:=C^2([0,a^{-1}], G), \,\,\,\,\,\,\, g(t) \longmapsto g(at).
}
induces a morphism of algebras ${\rm Lie}\,(\Gg_{00})$ and ${\rm Lie}\,((\Gg_{a})_{00})$, sending $\xi(t) \mapsto \xi(at)$.

\begin{lem}
The map $\phi_a$ preserves the degeneracy locus.
\end{lem}
\begin{proof}
Suppose that a solution $\Tt$ lies in the degeneracy locus, which we denote by $\Dd$. This implies that \eqref{eq: degen locus criterion} has a non-trivial solution $\xi(t)$ at $\Tt$. We claim that for any $a \in \R^+$, $\phi_a(\Tt)$ lies in the degeneracy locus $\Dd_{a^{-1}}$ for the equations defined on $[0,a^{-1}]$. To see this, note that the ODE \eqref{eq: degen locus criterion} has a non-trivial solution given by $\xi(at)$. In other words, the image $\phi_a(\Dd) \subset \Dd_{a^{-1}}$ for any $a \in \R^+$. On the other hand, the map $\phi_a$ has a smooth inverse given by $\phi_{a^{-1}}$ and a verbatim argument in the other direction shows that $\phi_{a^{-1}}(\Dd_{a^{-1}}) \subset \Dd$. In conclusion, we see that $\phi_a(\Dd) = \Dd_{a^{-1}}$ and thus $\phi_a$ maps the degeneracy locus to degeneracy locus.
\end{proof}

With this observation at hand, we will now define a homothetic action of $\R^+$ on $\Aa$. Assume first that $a <1$, so that $a^{-1} > 1$. Consider the restriction map
\eqst{
\beta_a: C^1([0,a^{-1}], \mf{g}) \longrightarrow C^1([0,1], \mf{g}), \,\,\,\,\,\,\, \Tt(t) \longmapsto \Tt(t)\,|_{\,[0,1]}.
}
Clearly, $\beta_a$ maps solutions to solutions. Composing $\beta_a$ with $\phi_a$, we get a map
\eqst{
\Phi_a: C^1([0,1], \mf{g}) \longrightarrow C^1([0,1], \mf{g}), \,\,\,\,\,\, \Tt \longmapsto \phi_a(\Tt)\,\vert_{\,[0,1]}.
}
Observe that $\Tt(t) \mapsto a\,\Tt\,(at)|_{\,[0,1]}$ is a homothety. To show that $\Phi_a$ has a smooth inverse, note that owing to Proposition \ref{prop: NS conserved quantity}, any solution to Nahm-Schmid equations, defined on $[0,1]$, can be uniquely extended to a solution on $[0,a^{-1}]$. This follows from the standard theory of existence and uniqueness of solutions to ODEs. Composing this extension map with $\phi_{a^{-1}}$ gives the inverse to $\Phi_a$. In particular, $\Phi_a$ is a diffeomorphism.

Since any solution $\Tt$ can be uniquely extended from $[0,1]$ to $[0,a^{-1}]$, and $\{T_i\}_{i=0}^3$ are continuous, it follows from the standard theory of existence and uniqueness of solutions to second order ODEs that any solution to \eqref{eq: degen locus criterion} on $[0,1]$, can be extended uniquely to a solution on $[0,a^{-1}]$. In other words, we have $\Phi_a(\Dd) = \Dd$.

Now suppose if $a > 1$, then $\Phi_a$ is given by composing $\phi_a$ with the extension map and the inverse is given by composition of restriction map $\beta_a$ with $\phi_{a^{-1}}$. The rest of the arguments are verbatim to the ones given above. In conclusion, the map $\Phi_a$ determines a homothetic action of $\R^+$ on $\Aa$, that preserves the space of solutions to Nahm-Schmid equations and also the degeneracy locus. The action commutes with that of $\Gg_{00}$ and ${\rm Sp}(1,\B)$. 

\begin{note}
Above lemma  implies that the $\R^+$-orbits of elements in the complement of the degeneracy locus, does not intersect $\Dd$.
\end{note}

\begin{thm}
The moduli space of solutions to the Nahm-Schmid equations, away from the degeneracy locus $\Dd$, is the total space of a Swann bundle over a paraquaternionic K\"ahler manifold. In other words,
\eqst{
\Mm_0 \, = \, \Uu(N_0) \longrightarrow N_0.
}
\end{thm}
\begin{proof}
Consider the open set $\mu^{-1, \ast}_{\sst \Aa}(0) = \mu^{-1}_{\sst \Aa}(0) \setminus \Dd$. Then, $\Mm_0:=\mu^{-1, \ast}_{\sst \Aa}(0)/ \Gg_{00}$ is a hypersymplectic manifold, which is an open set of $\Mm$. Moreover, since ${\rm Sp}(1,\B) \times \R^+$ action preserves $\mu^{-1, \ast}_{\sst \Aa}(0)$, the action descends to $\Mm_0$. As a result, $\Mm_0$ is topologically a metric cone over a split 3-Sasakian manifold $\Ss_0$
\eqst{
\Mm_0 \, = \, \Ss_0 \times \R^+.
}
Since the group ${\rm Sp}(1,\B)$ acts freely and properly on $\Ss_0$, by Theorem \ref{thm: pqk manifold as quotient of split 3-sasakian}, $N_0 := \Ss_0/{\rm Sp}(1,\B)$
is a paraquaternionic K\"ahler manifold.
\end{proof}

\bibliographystyle{abbrv}
\bibliography{references_nhsm}
\end{document}